\documentclass[12pt]{article}

 \usepackage{geometry}
\usepackage{mathrsfs}
\usepackage{graphicx, amsmath,latexsym,amsfonts,titlesec,amscd,epsfig,cite,color,
                                                                                   }
\usepackage{amssymb}
 \usepackage{autobreak}       
 \allowdisplaybreaks[4]   

\usepackage{indentfirst} 
\usepackage[amsmath,thmmarks]{ntheorem} 

\usepackage{extarrows}  
 \usepackage[sort&compress,numbers]{natbib}  
 \usepackage[bookmarks=true]{hyperref}
 \hypersetup{
    colorlinks=true,      
    citecolor=magenta,    
}
  \usepackage[capitalise]{cleveref}

 \usepackage[shortlabels]{enumitem}  
 \usepackage{lineno}          
\modulolinenumbers[5]  
\theoremstyle{definition} 
\theoremheaderfont{\normalfont\rmfamily\bf}
\theorembodyfont{\normalfont\rm } \theoremindent0em
\newtheorem{theorem}{\indent
                  Theorem}[section]
    \newtheorem{lemma}{\indent  Lemma} [section]
\newtheorem{corollary}{\indent Corollary} [section]

    \newtheorem{definition}{\indent  Definition} [section]
        \newtheorem{remark}{\indent  Remark}  

\theoremstyle{nonumberplain} 
\theoremheaderfont{\bf\rmfamily} \theorembodyfont{\normalfont \rm }
\theoremsymbol{\hfill $\Box$} 
\newtheorem{proof}{\indent Proof}

\def\loc{{\mathrm{loc}}}
\DeclareMathOperator*{\essinf}{ess\, inf}
\DeclareMathOperator*{\esssup}{ess\, sup}
\def\dint{\displaystyle\int}

\DeclareMathOperator*{\supp}{supp}

\numberwithin{equation}{section} 

\newcommand{\mathe}{\mathrm{e}}

\DeclareMathOperator*{\dini}{Dini}
\def\corrauth{\footnote{Corresponding author.
 }
 \stepcounter{footnote}
 }

\def\affil#1{${}^{#1}$}
\def\comma{$^{\textrm{,}}$}
\makeatletter
\renewcommand*{\author}[2][?]{
     \gdef\shortauthor{?}
     \gdef\@author{#2}
   \ifthenelse{\equal{#1}{?}}
     { \gdef\shortauthor{\let\comma=\empty \let\corrauth=\empty \renewcommand{\affil}[1]{} #2} }
     { \gdef\shortauthor{#1}}
}
\def\@setauthor{\begin{center}
  \sc  \@author
    \end{center}%
}
\makeatother

\title{\bf\large   Multilinear  fractional Calder\'{o}n-Zygmund operators with\\ Dini type kernel}

\author[J. WU and P. Zhang] {Jianglong Wu  and Pu Zhang\corrauth 
\\ 
{ 
\footnotesize\it Department of Mathematics, Mudanjiang Normal University, Mudanjiang, 157011, China}
}

\date{}

\begin{document}


\maketitle


\vspace{-4em}
\renewcommand\abstractname{}
\begin{abstract}
\setlength{\parindent}{0pt}\setlength{\parskip}{1.5explus0.5exminus0.2ex}%
\noindent
{\textbf{Abstract}:} In this paper, the main purpose   is to consider    a number of results concerning boundedness of multilinear  fractional Calder\'{o}n-Zygmund operators    with  kernels of mild regularity.
Let  $T_{\alpha}$ be a  multilinear fractional Calder\'{o}n-Zygmund operators of type $\omega(t)$ with  $\omega$ being nondecreasing and $\omega \in \dini(1)$.
The end-point weak-type estimates for multilinear operator $T_{\alpha}$ are obtained. Moreover, some boundedness properties of the multilinear  fractional  operators are also established on variable exponent Lebesgue spaces.


 { \bf AMS(2020) Subject Classification:}  \  42B20; 42B35   \ \
\end{abstract}

\section{Introduction and main results}

The multilinear Calder\'{o}n-Zygmund theory was first studied by Coifman and Meyer in \cite{coifman1975commutators,coifman1978commutateurs}. This theory was then further investigated by many authors in the last few decades,  see for example \cite{grafakos2002multilinear,lerner2009new,grafakos2014multilinear}, for the theory of multilinear Calder\'{o}n-Zygmund  operators with kernels satisfying the standard estimates. 
And multilinear fractional integral operators were first studied by Grafakos \cite{grafakos1992multilinear}, followed by Kenig and
Stein \cite{kenig1999multilinear} et al.
The importance of fractional integral operators is due to the fact that they 
have been widely used in various areas, such as potential analysis, harmonic analysis, and partial differential equations and so on\cite{wang2019multilinear}.
 The fractional Calder\'{o}n-Zygmund operators and related problems have been studied by a number of authors, see, for instance,  Lin and Lu \cite{lin2007boundedness},
Cruz-Uribe, Moen  and Van Nguyen \cite{cruz2019multilinear}, Wang and  Xu \cite{wang2019multilinear}, et al.


In 2009, Moen \cite{moen2009weighted} presented a weighted theory for multilinear fractional integral operators and maximal functions.
In 2014, Lu and Zhang \cite{lu2014multilinear} established  some boundedness results of  multilinear Calder\'{o}n-Zygmund operator  of type $\omega(t)$  and its commutators in variable Lebesgue spaces.
And recently,
Dalmasso et al. \cite{dalmasso2018effect}  proved boundedness results for integral operators of fractional type and their higher order commutators between weighted spaces, and the  kernels of such operators satisfy certain size condition and a Lipschitz type regularity, and the symbol of the commutator belongs to a Lipschitz class. They  also deal with commutators of fractional type operators with less regular kernels satisfying a H\"{o}rmander's type inequality.

Inspired by the works above, the main  goal of this paper is to consider  a number of results concerning boundedness of multilinear  fractional Calder\'{o}n-Zygmund operators $T_{\alpha}$   with  kernels of mild regularity. 
In particular, the corresponding conclusions can be found  in \cite{lu2014multilinear} when  $\alpha= 0$.

In what follows, let   $\mathbb{R}^{n}$ be an $n$-dimensional Euclidean space and $(\mathbb{R}^{n})^{m}= \mathbb{R}^{n} \times \cdots \times \mathbb{R}^{n}$ be an $m$-fold product space ($m\in  \mathbb{N}$).
We denote by $\mathscr{S}(\mathbb{R}^{n})$ the space of all Schwartz functions on  $\mathbb{R}^{n}$ and by $\mathscr{S}'(\mathbb{R}^{n})$ its dual space, the set of all tempered distributions on  $\mathbb{R}^{n}$.
And let $ C_{c}^{\infty}(\mathbb{R}^{n})$  denote the set of smooth functions with compact support in $\mathbb{R}^{n}$.


In \cite{yabuta1985generalizations}, in order to facilitate the study of  certain classes of pseudodifferential operators, Yabuta introduced certain  $\omega$-type Calder\'{o}n-Zygmund operators.
Let $\omega(t): [0,\infty) \to [0,\infty)$ be a non-negative and non-decreasing function with $0<\omega(1)<\infty$. For $a>0$,
\begin{enumerate}[leftmargin=2em,label=(\arabic*),itemindent=1.5em]  
\item  $\omega$ is said to satisfy the $\dini(a)$ condition, i.e. $\omega \in \dini(a)$,  if
\begin{align*}\label{equ:dini-condition-a}
|\omega|_{\dini(a)} &\triangleq \int_{0}^{1}  \omega^{a}(t)   \frac{\mathrm{d}t}{t} <\infty.
\end{align*}
\item      $\omega$ is said to satisfy the $\log^{m}-\dini(a)$ condition   if the following inequality holds
\begin{align*}\label{equ:dini-condition-am}
 \int_{0}^{1}  \omega^{a}(t) (1+\log t^{-1})^{m}   \frac{\mathrm{d}t}{t} <\infty,
\end{align*}
where $m\in \mathbb{Z}^{+}$.
\end{enumerate}

It is easy to check that the $\log^{m}-\dini(a)$ condition is stronger than the $\dini(a)$ condition,   and if $0 < a_{1} < a_{2} < \infty$, then $\dini(a_{1}) \subset \dini(a_{2})$.

In particular,  if   $\omega \in \dini(1)$, then
\begin{align*}
 \sum_{j=0}^{\infty} \omega(2^{-j}) &\approx \int_{0}^{1}  \omega(t)   \frac{\mathrm{d}t}{t} <\infty.
\end{align*}
And if $\omega\in \log-\dini(1)$,  that is
\begin{align*}\label{equ:dini-condition-m}
\dint_{0}^{1}\omega(t) (1+\log t^{-1})  \frac{\mathrm{d}t}{t} &< \infty,
\end{align*}
then $\omega \in \dini(1)$ and
\begin{align*}
 \sum_{k=1}^{\infty} k \omega(2^{-k}) &\approx \int_{0}^{1}  \omega(t) (1+\log t^{-1})  \frac{\mathrm{d}t}{t} <\infty.
\end{align*}


The following gives the definition of the  multilinear  fractional Calder\'{o}n-Zygmund  operators  of type $\omega(t)$.

 \begin{definition}\label{def:w-CZK}  
 Let   $0 \le \alpha < mn$.
 A locally integrable function $K_{\alpha}(x, y_{1},\dots , y_{m})$, defined away from the diagonal $x= y_{1} = \cdots = y_{m} $ in
$(\mathbb{R}^{n})^{m+1}$, is called an $m$-linear  Calder\'{o}n-Zygmund kernel  of type $\omega(t)$, if there exists a constant $A > 0$ such that the following conditions are satisfied.
\begin{enumerate}[leftmargin=2em,label=(\arabic*),itemindent=1.5em]  
\item  Size estimate:
 \begin{align} \label{equ:w-CZK-frac-size-estimate}
| K_{\alpha}(x,\vec{y})| &\le \dfrac{A}{(|x-y_{1}|+\cdots+|x-y_{m}|)^{mn-\alpha}}
\end{align}
for  all  $(x,y_{1},\dots , y_{m})\in (\mathbb{R}^{n})^{m+1}$ with $x\neq y_{j}$ for some $j\in\{1,2,\dots,m\}$.
\item    Smoothness estimate:  assume that  for each $j\in\{1,2,\dots,m\}$, there are regularity conditions
 \begin{align} \label{equ:w-CZK-frac-regularity-1}
|K_{\alpha}(x,\vec{y}) -K_{\alpha}(x',\vec{y}) |  &\le \dfrac{A }{\Big (\sum\limits_{j=1}^{m}|x-y_{j}| \Big)^{mn-\alpha}} \omega \Big( \frac{|x-x'|}{ |x-y_{1}|+\cdots+|x-y_{m}|} \Big)
\end{align}
whenever $|x-x'| \le \dfrac{1}{2}   \max\limits_{1\le j\le m} |x- y_{j}|$.
And  for each fixed $j$ with $1\le j \le m$,
 \begin{align} \label{equ:w-CZK-frac-regularity-2}
\begin{aligned}
|K_{\alpha}(x,y_{1},\dots,y_{j}, \dots , y_{m}) &-K_{\alpha}(x,y_{1},\dots,y_{j}', \dots , y_{m}) |  \\
 & \le \dfrac{A }{\Big (\sum\limits_{j=1}^{m}|x-y_{j}| \Big)^{mn-\alpha}} \omega \Big( \frac{|y_{j}-y_{j}'|}{ |x-y_{1}|+\cdots+|x-y_{m}|} \Big)
\end{aligned}
 \end{align}
whenever $|y_{j}-y_{j}'| \le \dfrac{1}{2}  \max\limits_{1\le j\le m} |x- y_{j}|$.
\end{enumerate}
\end{definition}

We say  $T_{\alpha}:\mathscr{S}(\mathbb{R}^{n})\times\cdots \times\mathscr{S}(\mathbb{R}^{n}) \to \mathscr{S}'(\mathbb{R}^{n})$
is an  $m$-linear fractional singular integral operator with an  $m$-linear fractional Calder\'{o}n-Zygmund kernel  of type $\omega(t)$, $K_{\alpha}(x,y_{1},\dots , y_{m})$, if
\begin{align} \label{equ:wm-frac-CZO}
T_{\alpha} (\vec{f})(x)  &=
 \dint_{(\mathbb{R}^{n})^{m}} K_{\alpha}(x,y_{1},\dots , y_{m}) \prod_{i=1}^{m}f_{i}(y_{i}) d\vec{y}
\end{align}
whenever $x\notin \bigcap_{j=1}^{m} \supp f_{j}$ and each $f_{j} \in C_{c}^{\infty}(\mathbb{R}^{n}), j=1,\dots,m$.

 \begin{definition}\label{def:wm-frac-CZO}
 Let $0<\alpha<mn $ and $T_{\alpha} $   be an $m$-linear fractional singular integral operator defined by \labelcref{equ:wm-frac-CZO}. Suppose that $1\le p_{_{1}},\dots, p_{_{m}} < \infty$ such that $1/p =  1/p_{_{1}}  + 1/p_{_{2}} +\cdots+1/p_{_{m}}$ and $1/q =  1/p  - \alpha/n >0$. Then $T_{\alpha} $ is called as an  $m$-linear  fractional Calder\'{o}n-Zygmund operator  of type $\omega$ (abbreviated to $m$-linear $\omega_{\alpha}$-CZO)   if the following conditions are satisfied:
\begin{enumerate}[leftmargin=2em,label=(\arabic*),itemindent=1.5em]  
\item   For some given numbers  $1<p_1,\dots,p_m<\infty$, 
    $T_{\alpha}$  maps  $L^{p_1}(\mathbb{R}^{n}) \times  \cdots\times L^{p_m}(\mathbb{R}^{n}) $ into $L^{q}(\mathbb{R}^{n})$.
\item     For some given numbers  $1\le p_1,\dots,p_m<\infty$ and $\min\limits_{1\le j\le m} \{p_{j}\}=1$, 
    $T_{\alpha}$  maps  $L^{p_1}(\mathbb{R}^{n}) \times  \cdots\times L^{p_m}(\mathbb{R}^{n}) $ into $L^{q,\infty}(\mathbb{R}^{n})$.
\end{enumerate}
\end{definition}


\begin{remark}
\begin{enumerate}[leftmargin=2em,label=(\alph*),itemindent=1.5em]  
\item  Obviously, the $m$-linear $\omega_{\alpha}$-CZO is exactly the multilinear  Calder\'{o}n-Zygmund operator studied by Grafakos and Torres in \cite{grafakos2002multilinear} when $\alpha=0$ and $\omega(t)=t^{\varepsilon}$ for some $\varepsilon >0$.
\item    When $\alpha=0$, the $m$-linear $\omega_{\alpha}$-CZO is exactly the multilinear  Calder\'{o}n-Zygmund operator studied by Lu and Zhang in \cite{lu2014multilinear}.
\item     Using size estimate \eqref{equ:w-CZK-frac-size-estimate} with $\alpha=0$,  from the  proof of Lemma 2 in   \cite{grafakos2002multilinear}, the following condition can be obtained
\begin{equation}\label{inequ:(m-1)-CZK}
 \dint_{\mathbb{R}^{n}} \dfrac{1}{(|x-y_{1}|+\cdots+|x-y_{m}|)^{nm}} d y_{m}\le \frac{C}{(|x-y_{1}|+\cdots+|x-y_{m-1}|)^{n(m-1)}}.
\end{equation}
By symmetry, it is also true if we freeze any other variable in $K$ instead of $y_{m}$.
\end{enumerate}
\end{remark}

%

Throughout this paper, the letter $C$  always stands for a constant  independent of the main parameters involved and whose value may differ from line to line.
A cube $Q \subset\mathbb{R}^{n} $ always means a cube whose sides are parallel to the coordinate axes and denote its side length by $l(Q)$.
For some $t>0$, the  notation  $tQ$ stands   for  the cube with the same center as  $Q$ and with side length  $l(tQ)=t l(Q)$.
 Denote by $|S|$ the Lebesgue measure  and  by $\chi_{_{\scriptstyle S}}$ the characteristic function  for a measurable set  $S\subset\mathbb{R}^{n}$. $B(x,r)$ means the ball centered at $x$ and of radius $r$, and $B_{0}=B(0,1)$.
 $X \approx Y$ means there is a constant $C>0$ such that $C^{-1}Y \le X \le C Y$.
For any index $1< q(x)< \infty$, we denote by $q'(x)$ its conjugate index,
namely, $q'(x)=\frac{q(x)}{q(x)-1}$. And we will occasionally use the notational $\vec{f}=(f_{1},\dots , f_{m})$, $T(\vec{f})=T(f_{1},\dots , f_{m})$, $d\vec{y}=dy_{1}\cdots  dy_{m}$ and $(x,\vec{y})=(x,y_{1},\dots , y_{m})$ for convenience.
For a set $E$ and a positive integer $m$, we will use the notation $(E)^{m}=\underbrace{E\times \cdots \times E}_{m}$ sometimes.

\subsection{Boundedness of $m$-linear $\omega_{\alpha}$-CZO}

The first result for multilinear fractional operators $T_{\alpha}$ with multilinear fractional Calder\'{o}n-Zygmund kernel of type $\omega$  is the following end-point weak-type estimates on the product of Lebesgue spaces. And
the Calder\'{o}n-Zygmund decomposition is the key tool used  in obtaining endpoint weak type results for the $m$-linear $\omega_{\alpha}$-CZO.
\begin{theorem} \label{thm:dini-multi-fract-endpoint}
Let  $0<\alpha<mn $ and $T_{\alpha}$ be an   $m$-linear $\omega_{\alpha}$-CZO with $\omega \in \dini(1)$. Suppose that for some $1\le p_{_{1}},p_{_{2}},\dots,p_{_{m}} \le \infty$ and some $0<p,q<\infty$ with
\begin{align*}
 1/p =  1/p_{_{1}}  + 1/p_{_{2}} +\cdots+1/p_{_{m}},  &\qquad 1/q =  1/p  - \alpha/n,
\end{align*}
$T_{\alpha}$  maps $L^{p_1}(\mathbb{R}^{n}) \times  \cdots\times L^{p_m}(\mathbb{R}^{n}) $ into  $L^{q,\infty}(\mathbb{R}^{n})$.
Then  $T_{\alpha}$  can be extended to a bounded operator from the $m$-fold product  $L^{1}(\mathbb{R}^{n}) \times  \cdots\times L^{1}(\mathbb{R}^{n}) $ into  $L^{\frac{n}{mn-\alpha},\infty}(\mathbb{R}^{n})$.
Moreover, there is a constant $C _{m, n,|\omega|_{\dini(1)}}$ (that depends only on the parameters indicated) such that
\begin{align*}
 \|T_{\alpha}\|_{L^{1}\times  \cdots\times L^{1} \to L^{\frac{n}{mn-\alpha},\infty}}&\le  C _{m, n,|\omega|_{\dini(1)}} (A+ \|T_{\alpha}\|_{L^{p_1} \times  \cdots\times L^{p_m} \to L^{q,\infty}}),
\end{align*}
where $A$ is the constant appearing in \labelcref{equ:w-CZK-frac-size-estimate,equ:w-CZK-frac-regularity-1,equ:w-CZK-frac-regularity-2}.
\end{theorem}

\begin{remark}
\begin{enumerate}[leftmargin=2em,label=(\roman*),itemindent=1.5em]  
\item  When $\alpha=0$,  \cref{thm:dini-multi-fract-endpoint}  was proved in \cite{lu2014multilinear}.
\item   When $\alpha=0$ and $\omega(t) = t^{\varepsilon}$ for some $\varepsilon>0$,  \cref{thm:dini-multi-fract-endpoint}  was proved in \cite{grafakos2002multilinear}.
\item   For the bilinear case, \cref{thm:dini-multi-fract-endpoint} was proved in \cite{maldonado2009weighted} when $\alpha=0$, $\omega$ is concave and $\omega \in \dini(1/2)$.
\end{enumerate}
\end{remark}

To state the weighted norm inequalities for the multilinear  fractional Calder\'{o}n-Zygmund  operators of type $\omega$, we first recall some notation and definition on weights.
The class of $A_{\vec{P}}$ can be found in \cite{lerner2009new}, and the class of $A_{\vec{P},q}$ can be found  in \cite{moen2009weighted,chen2010weighted}.
\begin{definition}[Multiple weights]    \
For $m$ exponents $ p_{j}$,  let $\vec{P} = (p_{1},\dots ,p_{m})$ and $\frac{1}{p}= \frac{1}{p_{1}}+ \cdots + \frac{1}{p_{m}}$ with $1\le p_{j} < \infty~ (j = 1,\dots,m)$.  Given  $\vec{w} = (w_{1},\dots ,w_{m})$ with nonnegative function $w_{1},\dots ,w_{m} \in \mathbb{R}^{n}$.
\begin{enumerate}[ label=(\arabic*),itemindent=1em] 
\item  (Class of $A_{\vec{P}}$)  We say that  $\vec{w}$ satisfies the  $A_{\vec{P}}$ condition, i.e. $\vec{w} \in A_{\vec{P}}$, if
\begin{align*}
\sup_{Q}   \Big(\frac{1}{|Q|} \int_{Q}  u_{\vec{w}}(x)  dx  \Big)^{1/p}  \prod_{j=1}^{m} \Big(\frac{1}{|Q|} \int_{Q}  \big( w_{j} (x) \big)^{1-p'_{j}}   dx \Big)^{ 1/p'_{j}} &<\infty,
\end{align*}
where the supremum is taken over all cubes $Q \subset \mathbb{R}^{n}$,  $u_{\vec{w}} = \prod_{j=1}^{m}  w_{j}^{p/p_{j}}$,   and
$ \Big(\frac{1}{|Q|} \int_{Q}  \big( w_{j} (x) \big)^{1-p'_{j}}   dx \Big)^{ 1/p'_{j}}$
in the case $p_{j} = 1$   is understood as $(\inf_{Q} w_{j})^{-1}$.
\item  (Class of $A_{\vec{P},q}$)\ Let $q$  be a number $1/m < p \le q < \infty$. We say that  $\vec{w}$ satisfies the  $A_{\vec{P},q}$ condition if
\begin{align*}
\sup_{Q}   \Big(\frac{1}{|Q|} \int_{Q} \big(v_{\vec{w}}(x)\big)^{q} dx  \Big)^{1/q}  \prod_{j=1}^{m} \Big(\frac{1}{|Q|} \int_{Q}  \big( w_{j} (x)  \big)^{-p'_{j}}   dx \Big)^{ 1/p'_{j}} &<\infty,
\end{align*}
where the supremum is taken over all cubes $Q$ in $\mathbb{R}^{n}$, $q>0$, $v_{\vec{w}} = \prod_{j=1}^{m}  w_{j}$,  and  $\Big(\frac{1}{|Q|} \int_{Q}  \big( w_{j} (x)  \big)^{-p'_{j}}   dx \Big)^{ 1/p'_{j}}$    is understood as $(\inf_{Q} w_{j})^{-1}$ when $p_{j} = 1$.
\end{enumerate}
\end{definition}

In addition, For $0 < p < \infty$ and $w \in A_{\infty}$, denote by $L^{p} (w)= L^{p} (\mathbb{R}^{n},w)$ the collection of all functions $f$ satisfying
\begin{align*}
\|f\|_{L^{p} (w)} &= \Big( \dint_{\mathbb{R}^{n}} |f(x)|^{p} w(x) \mathrm{d}x \Big)^{1/p} < \infty.
\end{align*}
And, denote by $L^{p,\infty} (w)= L^{p,\infty} (\mathbb{R}^{n},w)$  the weak space with norm
\begin{align*}
\|f\|_{L^{p,\infty} (w)} &= \sup_{t>0} t w\big(\{x\in\mathbb{R}^{n}: |f(x)|>t\} \big)^{1/p},
\end{align*}
where $w(E)=\int_{E} w(x) \mathrm{d}x$  for a measurable set $E \subset \mathbb{R}^{n}$.

Now, we state the multiple-weighted norm inequalities and weak-type estimates for the multilinear  fractional Calder\'{o}n-Zygmund  operators of type $\omega$.

\begin{theorem} \label{thm:dini-multi-fract-czo-weight}
Let  $0<\alpha<mn $ and $T_{\alpha}$ be an   $m$-linear $\omega_{\alpha}$-CZO with $\omega \in \dini(1)$. Given  $\vec{P} = (p_{1},\dots ,p_{m})$ and $\frac{1}{p}= \frac{1}{p_{1}}+ \cdots + \frac{1}{p_{m}}$ with $1\le p_{j} < \infty~ (j = 1,\dots,m)$.
Suppose that 
$1/q =  1/p  - \alpha/n >0$ and $\vec{w} \in A_{\vec{P},q}$.
\begin{enumerate}[leftmargin=2em,label=(\roman*),itemindent=1.5em]  
\item  If   $1< p_{_{j}} < \infty$ for all $j=1,\dots,m$, then
\begin{align*}
  \|T_{\alpha}(\vec{f})\|_{L^{q}(v_{\vec{w}}^{q})}  &\le C  \prod_{j=1}^{m}\|f_{j}\|_{L^{p_{j}}(w_{j}^{p_{j}})}.
\end{align*}
\item  If  $1 \le p_{_{j}} < \infty$ for all $j=1,\dots,m$, and at least one $p_{_{j}} =1$ for some $j=1,\dots,m$, then
\begin{align*}
 \|T_{\alpha}(\vec{f})\|_{L^{q,\infty}(v_{\vec{w}}^{q})}  &\le  C \prod_{j=1}^{m}\|f_{j}\|_{L^{p_{j}}(w_{j}^{p_{j}})}.
\end{align*}
\end{enumerate}
\end{theorem}

As a consequence of  the theorem above the following result is obtained. The proof is left to the reader since it is simple.
\begin{corollary} \label{cor:dini-multi-fract-czo}
Let  $0<\alpha<mn $ and $T_{\alpha}$ be an   $m$-linear $\omega_{\alpha}$-CZO with $\omega \in \dini(1)$. Suppose that 
$1/q =  1/p_{_{1}}  + 1/p_{_{2}} +\cdots+1/p_{_{m}} - \alpha/n >0$.
\begin{enumerate}[leftmargin=2em,label=(\roman*),itemindent=1.5em]  
\item  If   $1< p_{_{j}} < \infty$ for all $j=1,\dots,m$, then
\begin{align*}
  \|T_{\alpha}(\vec{f})\|_{L^{q}(\mathbb{R}^{n})}  &\le C  \prod_{j=1}^{m}\|f_{j}\|_{L^{p_{j}}(\mathbb{R}^{n})}.
\end{align*}
\item  If  $1 \le p_{_{j}} < \infty$ for all $j=1,\dots,m$, and at least one $p_{_{j}} =1$ for some $j=1,\dots,m$, then
\begin{align*}
 \|T_{\alpha}(\vec{f})\|_{L^{q,\infty}(\mathbb{R}^{n})}  &\le  C \prod_{j=1}^{m}\|f_{j}\|_{L^{p_{j}}(\mathbb{R}^{n})}.
\end{align*}
\end{enumerate}
\end{corollary}

\begin{remark}
\begin{enumerate}[leftmargin=2em,label=(\roman*),itemindent=1.5em]  
\item  When $\alpha=0$,  \cref{thm:dini-multi-fract-czo-weight}  was proved in \cite{lu2014multilinear}.
\item   When $\alpha=0$ and $\omega(t) = t^{\varepsilon}$ for some $\varepsilon>0$,  \cref{thm:dini-multi-fract-czo-weight}  was proved in \cite{grafakos2002multilinear}.
\item  when $\alpha=0$, $\omega$ is concave and $\omega \in \dini(1/2)$, the first part of \cref{thm:dini-multi-fract-czo-weight}  was proved in   \cite{maldonado2009weighted}.
\end{enumerate}
\end{remark}

\subsection{On variable exponent Lebesgue spaces}

In this section, we will study the boundedness properties of  $m$-linear $\omega_{\alpha}$-CZO  with mild regularity on variable exponent Lebesgue spaces. We first recall some definitions and notations.


 \begin{definition} \label{def.variable-Lebesgue}
  Let ~$ q(\cdot): \mathbb{R}^{n}\to[1,\infty)$ be a measurable function.
\begin{enumerate}[leftmargin=2em,label=(\roman*),itemindent=1.5em]  
\item  The variable exponent Lebesgue spaces $L^{q(\cdot)}(\mathbb{R}^{n})$ is defined by
  $$ L^{q(\cdot)}(\mathbb{R}^{n})=\{f~ \mbox{is measurable function}:  F_{q}(f/\eta)<\infty ~\mbox{for some constant}~ \eta>0\}, $$
  where $F_{q}(f):=\int_{\mathbb{R}^{n}} |f(x)|^{q(x)} \mathrm{d}x$ is a convex functional modular. The Lebesgue space $L^{q(\cdot)}(\mathbb{R}^{n})$ is a Banach function space with respect to the Luxemburg type norm
  \begin{equation*}
   \|f\|_{L^{q(\cdot)}(\mathbb{R}^{n})}=\inf \Big\{ \eta>0:  F_{q}(f/\eta)=\int_{\mathbb{R}^{n}} \Big( \frac{|f(x)|}{\eta} \Big)^{q(x)} \mathrm{d}x \le 1 \Big\}.
\end{equation*}
\item    The space $L_{\loc}^{q(\cdot)}(\mathbb{R}^{n})$ is defined by
  $$ L_{\loc}^{q(\cdot)}(\mathbb{R}^{n})=\{f ~\mbox{is measurable}: f\in L^{q(\cdot)}(E) ~\mbox{for all compact  subsets}~ E\subset \mathbb{R}^{n}\}.  $$
  \item  The weighted Lebesgue space $L_{w}^{q(\cdot)}(\mathbb{R}^{n})$ is defined by as the set of all measurable functions for which $$\|f\|_{L^{q(\cdot)}_{w}(\mathbb{R}^{n})}=\|w f\|_{L^{q(\cdot)}(\mathbb{R}^{n})}<\infty.$$
\end{enumerate}
\end{definition}

Next we define some classes of variable  exponent functions. Given a function $f\in L_{\loc}^{1}(\mathbb{R}^{n})$, the Hardy-Littlewood maximal operator $M$ is defined by
$$Mf(x)= \sup_{Q\ni x} \dfrac{1}{|Q| } \int_{Q}  |f(y)| \mathrm{d}y. $$

\begin{definition} \label{def.variable-exponent}\ \
 \ Given a measurable function $q(\cdot)$ defined on $\mathbb{R}^{n}$. For $E\subset \mathbb{R}^{n}$, we write
$$
q_{-}(E):=\essinf_{x\in E} q(x),\ \
q_{+}(E):= \esssup_{x\in E} q(x),$$
and write $q_{-}(\mathbb{R}^{n}) = q_{-}$ and $q_{+}(\mathbb{R}^{n}) = q_{+}$ simply.

\begin{list}{}{}
\item[(i)]\   $q'_{-}=\essinf\limits_{x\in \mathbb{R}^{n}} q'(x)=\frac{q_{+}}{q_{+}-1},\ \ q'_{+}= \esssup\limits_{x\in \mathbb{R}^{n}} q'(x)=\frac{q_{-}}{q_{-}-1}.$

\item[(ii)]\ Denote by $\mathscr{P}_{0}(\mathbb{R}^{n})$
the set of all measurable functions $ q(\cdot): \mathbb{R}^{n}\to(0,\infty)$ such that
$$0< q_{-}\le q(x) \le q_{+}<\infty,\ \ x\in \mathbb{R}^{n}.$$

\item[(iii)]\ Denote by $\mathscr{P}_{1}(\mathbb{R}^{n})$
the set of all measurable functions $ q(\cdotp): \mathbb{R}^{n}\to[1,\infty)$ such that
$$1\le q_{-}\le q(x) \le q_{+}<\infty,\ \ x\in \mathbb{R}^{n}.$$

\item[(iv)]\ Denote by $\mathscr{P}(\mathbb{R}^{n})$ the set of all measurable functions $ q(\cdot): \mathbb{R}^{n}\to(1,\infty)$ such that
$$1< q_{-}\le q(x) \le q_{+}<\infty,\ \ x\in \mathbb{R}^{n}.$$

\item[(v)]\  The set $\mathscr{B}(\mathbb{R}^{n})$ consists of all  measurable functions  $q(\cdot)\in\mathscr{P}(\mathbb{R}^{n})$ satisfying that the Hardy-Littlewood maximal operator $M$ is bounded on $L^{q(\cdot)}(\mathbb{R}^{n})$.

\end{list}
\end{definition}

\begin{definition} [$\log$-H\"{o}lder continuity] \label{def.log-holder} \ \
 \ Let $q(\cdot)$ be a real-valued function on  $\mathbb{R}^{n}$.
\begin{list}{}{}
\item[(i)]\    Denote by  $\mathscr{C}^{\log}_{loc}(\mathbb{R}^{n})$
the set of all local  $\log$-H\"{o}lder continuous functions $q(\cdotp)$ which satisfies
 \begin{equation*}    
 |q(x)-q(y)| \le \frac{-C}{\ln(|x-y|)},  \ \ \  |x-y|\le 1/2,\ x,y \in \mathbb{R}^{n},
\end{equation*}
where $C$ denotes a universal positive constant that may differ from line to line, and  $C$ does not depend on $x, y$.
\item[(ii)]\  The set $\mathscr{C}^{\log}_{\infty}(\mathbb{R}^{n})$ consists of all  $\log$-H\"{o}lder continuous functions $ q(\cdot)$ at infinity   satisfies
\begin{equation*}
 |q(x)-q_{\infty}| \le \frac{C_{\infty}}{\ln(\mathe+|x|)},  \ \ \ x \in \mathbb{R}^{n},
\end{equation*}
 where $q_{\infty}=\lim\limits_{|x|\to \infty}q(x)$.
\item[(iii)]\  Denote by  $\mathscr{C}^{\log}(\mathbb{R}^{n}):=\mathscr{C}^{\log}_{loc}(\mathbb{R}^{n})\cap \mathscr{C}^{\log}_{\infty}(\mathbb{R}^{n})$ the set of all global $\log$-H\"{o}lder continuous functions $ q(\cdot)$.
\end{list}
\end{definition}

\begin{remark}
\begin{enumerate}[leftmargin=2em,label=(\alph*),itemindent=1.5em]  
\item  The $\mathscr{C}^{\log}_{\infty}(\mathbb{R}^{n})$ condition is equivalent to the uniform continuity condition
\begin{equation*}  
  |q(x)-q(y)| \le \frac{C}{\ln(\mathe+|x|)},  \ \ \  |y|\ge|x|,\ x,y \in \mathbb{R}^{n}.
\end{equation*}
The $\mathscr{C}^{\log}_{\infty}(\mathbb{R}^{n})$ condition was originally defined in this form in \cite{cruz2003maximal}.
\item   In what follows, we denote $\mathscr{C}^{\log}(\mathbb{R}^{n}) \cap\mathscr{P}(\mathbb{R}^{n})$ by  $\mathscr{P}^{\log}(\mathbb{R}^{n})$.
\end{enumerate}
\end{remark}

The theory of function spaces with variable exponent were first studied by Orlicz \cite{orlicz1931konjugierte}, and it has been intensely investigated in the past twenty years since some elementary properties were established by Kov{\'a}{\v{c}}ik  and R{\'a}kosn{\'\i}k in \cite{kovavcik1991spaces}, and  
because of its connection with the study of variational integrals and partial differential equations with non-standard growth conditions (see, for instance, \cite{acerbi2005gradient, harjulehto2008minimizers,sanchon2009entropy}). In 2003, Diening and R\r{u}\u{z}i\u{c}ka \cite{diening2003calderon} studied the Calder\'{o}n-Zygmund operators on variable exponent Lebesgue spaces and gave some applications to problems related to fluid dynamics. In 2006, by applying the theory of weighted norm inequalities and extrapolation, Cruz-Uribe et al. \cite{cruz2006theboundedness} showed that many classical operators in harmonic analysis are bounded on the variable exponent Lebesgue space. For more information on function spaces with variable exponent, we refer to \cite{diening2011lebesgue,cruz2013variable}.

For $m$-linear $\omega_{\alpha}$-CZO, we have the following result.
\begin{theorem} \label{thm:dini-multi-fract-czo-vari}
Let  $0<\alpha<mn $ and $T_{\alpha}$ be an   $m$-linear $\omega_{\alpha}$-CZO with $\omega \in \dini(1)$. Given  $\frac{1}{p(\cdot)}= \frac{1}{p_{1}(\cdot)}+ \cdots + \frac{1}{p_{m}(\cdot)}$ with $  p(\cdot), p_{j}(\cdot) \in \mathscr{P}^{\log}(\mathbb{R}^{n})~ (j = 1,\dots,m)$.
Suppose that $ 0< \frac{1}{q(\cdot)} =  \frac{1}{p(\cdot)}  - \frac{\alpha}{n}  <1$. Then there exists a positive constant $C$ such that
\begin{align*}
  \|T_{\alpha}(\vec{f})\|_{L^{q(\cdot)}(\mathbb{R}^{n} )}  &\le C  \prod_{j=1}^{m}\|f_{j}\|_{L^{p_{j}(\cdot)}(\mathbb{R}^{n})}.
\end{align*}
\end{theorem}

\section{Notation and preliminaries}

\subsection{ Sharp maximal function and $A_{p}$ weights}

The following  concepts are needed.

Let $f$ be a locally integral function defined on  $\mathbb{R}^{n}$. Denote by $M$  the usual Hardy-Littlewood maximal operator, for  a cube $Q \subset \mathbb{R}^{n}$ and $\delta>0$,   the maximal functions $M_{\delta}$ is defined by
\begin{align*}
M_{\delta} (f)(x)   &= \big[ M(|f|^{\delta}) (x)  \big]^{1/\delta}  = \Big(\sup\limits_{Q\ni x}  \dfrac{1}{|Q|}  \dint_{Q} |f(y)|^{\delta} \mathrm{d} y \Big)^{1/\delta}.
\end{align*}
Let  $M^{\sharp}$ be  the standard sharp  maximal function of Fefferman and Stein\cite{fefferman1972h}, that is
\begin{align*}
M^{\sharp}f (x)     &= \sup_{Q\ni x} \inf_{c} \dfrac{1}{|Q|} \int_{Q} |f(y)-c| \mathrm{d}y  \approx  \sup_{Q\ni x} \dfrac{1}{|Q|} \int_{Q} |f(y)-f_{Q}| \mathrm{d}y,
\end{align*}
where, as usual, $f_{Q}$ denotes the average of $f$ over $Q$, and the supremum is taking over all the cubes $Q$ containing the point $x$.
 The operator  $M_{\delta}^{\sharp}$  is defined by  $M_{\delta}^{\sharp}f (x)    = \big[ M^{\sharp}(|f|^{\delta})(x) \big]^{1/\delta}$.

Let $w$ be a nonnegative locally integrable function defined in $\mathbb{R}^{n}$.
\begin{enumerate}[leftmargin=2em,label=(\roman*),itemindent=1.5em]  
\item  For   $1<p<  \infty$,  we say that $w$  is in the Muckenhoupt class $A_{p}$  (namely, $w\in A_{p}$), if there exists a constant $C>0$ (depending on the $A_{p}$ constant of $w$) such that for any cube $Q$, there has
\begin{align*}
  \left(\dfrac{1}{|Q|} \int_{Q}  w(x) \mathrm{d}x \right)  \left(\dfrac{1}{|Q|} \int_{Q}  w(x)^{1-p'} \mathrm{d}x \right)^{p-1} \le C.
\end{align*}
\item  We say that $w\in A_{1}$ if there exists a constant $C>0$ (depending on the $A_{1}$ constant of $w$) such that $Mw(x) \le C w(x)$ almost everywhere $x\in \mathbb{R}^{n}$.
\item  Define the $A_{\infty}$ by  $A_{\infty} = \bigcup_{p\ge 1} A_{p}$.
\end{enumerate}
  See \cite{garcia1985weighted,stein1993harmonic} or \cite{grafakos2014classical}(Chapter 7) for more information about the Muckenhoupt weight class $A_{p}$.

The following relationships between $ M_{\delta}^{\sharp}$ and $M_{\delta}$ to be used is a version of the classical ones due to Fefferman and
Stein \cite{fefferman1972h} ( see also \cite{journe1983calderon} or P. 1228 in \cite{lerner2009new}).
\begin{lemma} \label{lem:relation-sharp-maximal-weight}
Let  $0<p,\delta< \infty$ and $w$ be any   $A_{\infty}$-weight.
\begin{enumerate}[leftmargin=2em,label=(\roman*),itemindent=1.5em]  
\item   Then there exists a constant $C>0$  (depending on the $A_{\infty}$ constant of $w$),   such that the inequality
\begin{align*}
 \int \big( M_{\delta} (f)(x) \big)^{p} w(x) \mathrm{d}x  &\le C \int \big( M_{\delta}^{\sharp}(f)(x) \big)^{p} w(x) \mathrm{d}x
\end{align*}
holds for any function $f$ for which the left hand side is finite.
\item  Similarly, there exists another constant $C > 0$  (depending on the $A_{\infty}$ constant of $w$), such that
\begin{align*}
 \| M_{\delta} (f)\|_{L^{p,\infty}(w)}  &\le  C  \|M_{\delta}^{\sharp}(f)\|_{L^{p,\infty}(w)}
\end{align*}
holds for any function $f$ for which the left hand side is finite.
\end{enumerate}
\end{lemma}

\subsection{Some auxiliary lemmas}

 In this part we   state some auxiliary propositions  and lemmas which will be needed for proving  our main theorems. And we only state partial results  we need.

 \begin{lemma}\quad \label{lem.variable-property-B}
Let $ p(\cdot)\in \mathscr{P}(\mathbb{R}^{n})$.
\begin{enumerate}[leftmargin=2em,label=(\arabic*),itemindent=1.5em]  
 \item \ If $ p(\cdot)\in \mathscr{C}^{\log}(\mathbb{R}^{n})$, 
 then we have $ p(\cdot)\in \mathscr{B}(\mathbb{R}^{n})$.
\item  (see Lemma 2.3 in \cite{cruz2014variable})
The  following conditions are equivalent:
\begin{enumerate}[label=(\roman*),align=left,itemindent=1em]  
 \item \   $ p(\cdot)\in \mathscr{B}(\mathbb{R}^{n})$,
 \item \   $p'(\cdot)\in \mathscr{B}(\mathbb{R}^{n})$.
  \item \   $ p(\cdot)/p_{0}\in \mathscr{B}(\mathbb{R}^{n})$ for some $1<p_{0}<p_{-}$,
 \item \   $ (p(\cdot)/p_{0})'\in \mathscr{B}(\mathbb{R}^{n})$ for some $1<p_{0}<p_{-}$.
\end{enumerate}
\end{enumerate}
\end{lemma}

 The  first part in   \cref{lem.variable-property-B} is independently due to Cruz-Uribe et al. \cite{cruz2003maximal} and to Nekvinda\cite{nekvinda2004hardy}  respectively. The second of  \cref{lem.variable-property-B}  belongs to  Diening\cite{diening2005maximalf}~(see Theorem 8.1 or Theorem 1.2 in  \cite{cruz2006theboundedness}).

%

The following gives the  generalized H\"{o}lder's inequality.
\begin{lemma} [generalized H\"{o}lder's inequality]\label{lem.holder-inequality} \
\begin{enumerate}[leftmargin=2em,label=(\arabic*),itemindent=1.5em]  
\item \ \ Let $p(\cdotp),q(\cdotp),r(\cdotp)\in  \mathscr{P}_{0}(\mathbb{R}^{n})$ satisfy the condition
\[
\dfrac{1}{r(x)} = \dfrac{1}{p(x)} + \dfrac{1}{q(x)} \qquad \mbox{for a.e.} \ x\in \mathbb{R}^{n}.
\]
\begin{enumerate}[label=(\roman*),align=left,itemindent=1em]  
\item Then, for all $f \in L^{{p(\cdotp)}}(\mathbb{R}^{n})$ and $g\in L^{{q(\cdotp)}}(\mathbb{R}^{n})$, one has
\begin{align} \label{equ:holder-2}
\|fg\|_{r(\cdotp)} &\le C\|f\|_{p(\cdotp)}  \|g\|_{q(\cdotp)}.
\end{align}
\item   When $r(\cdotp)=1$, then $p'(\cdotp) = q(\cdotp)$, hence, for all $f \in L^{{p(\cdotp)}}(\mathbb{R}^{n})$ and $g\in L^{{p'(\cdotp)}}(\mathbb{R}^{n})$, one has
\begin{align}\label{equ:holder-1}
\int_{\mathbb{R}^{n}}|fg|  &\le C\|f\|_{p(\cdotp)}  \|g\|_{p'(\cdotp)}.
\end{align}
\end{enumerate}
\item The generalized H\"{o}lder's inequality in Orlicz space:\
\begin{enumerate}[label=(\roman*),align=left,itemindent=1em]  
\item Let $r_1,\dots,r_m\ge 1$ with $\frac{1}{r}=\frac{1}{r_1}+\cdots+\frac{1}{r_m}$ and $Q$ be a cube in $\mathbb{R}^{n}$. Then
\begin{align*}
\frac{1}{|Q|}\int_{Q}|f_{1}(x)\cdots f_{m}(x)g(x)| dx &\le C\|f_{1}\|_{\exp L^{r_{1}},Q} \cdots \|f_{m}\|_{\exp L^{r_{m}},Q} \|g\|_{L(\log L)^{1/r},Q}.
\end{align*}
\item   Let $t\ge 1$, then
\begin{align}\label{equ:holder-4}
\frac{1}{|Q|}\int_{Q}|f(x)g(x)| dx &\le C\|f\|_{\exp L^{t},Q}   \|g\|_{L(\log L)^{1/t},Q}.
\end{align}
\end{enumerate}
\item \ Let $q(\cdotp),q_{1}(\cdot),\dots, q_{m}(\cdot)\in  \mathscr{P}(\mathbb{R}^{n})$ satisfy the condition
\[
\dfrac{1}{q(x)} = \dfrac{1}{q_{1}(x)}+\cdots + \dfrac{1}{q_{m}(x)} \qquad \mbox{for a.e.} \ x\in \mathbb{R}^{n}.
\]
Then, for any $f_{j} \in L^{{q_{j}(\cdotp)}}(\mathbb{R}^{n})$ , $j=1,\dots,m$, one has
\begin{align*} 
\|f_{1}\cdots f_{m} \|_{q(\cdotp)} &\le C\|f_{1}\|_{q_{1}(\cdot)} \cdots \|f_{m}\|_{q_{m}(\cdot)}.
\end{align*}
\end{enumerate}
\end{lemma}

In     \cref{lem.holder-inequality}, the  first part is known as  the generalized H\"{o}lder's inequality on variable exponent Lebesgue spaces, and the proof can   be found in \cite{kovavcik1991spaces}(see also P.27-30 in \cite{cruz2013variable} or P.81-82, Lemma 3.2.20 in \cite{diening2011lebesgue}); the  second  part    is  generalized H\"{o}lder's inequality in Orlicz space (for  details and  the more general cases see \cite{perez2002sharp,perez1995endpoint,lerner2009new});
and the third part  see Lemma 9.2 in \cite{lu2014multilinear}.

The following inequalities are also necessary (see (2.16) in \cite{lerner2009new} or Lemma 2.3 in \cite{xue2013weighted} or Lemma 4.6 in \cite{lu2014multilinear}  or page 485 in \cite{garcia1985weighted}).

\begin{lemma}[Kolmogorov's inequality] \label{lem:kolmogorov}
Let  $0<p<q<\infty $, cube $Q \subset \mathbb{R}^{n}$. Using $L^{q,\infty}(Q) $ denotes   the weak space with norm $\|f\|_{L^{q,\infty}(Q)} = \sup\limits_{t>0} t|\{ x\in Q: |f(x)|>t\}|^{1/q} $.
\begin{enumerate}[leftmargin=2em,label=(\roman*),itemindent=1.5em]  
\item  Then there is a positive constant $C=C_{p,q}$ such that for any measurable funcction $f$ there has
\begin{align*}
 |Q|^{-1/p} \|f\|_{L^{p}(Q)}  &\le C |Q|^{-1/q} \|f\|_{L^{q,\infty}(Q)}.
\end{align*}
\item  If $0<\alpha <n$ and   $1/q = 1/p - \alpha/n$. Then there is a positive constant $C=C_{p,q}$ such that  for any measurable function $f$ there has
\begin{align*}
\|f\|_{L^{p}(Q)}  &\le  C |Q|^{\alpha/n} \|f\|_{L^{q,\infty}(Q)}.
\end{align*}
\end{enumerate}
\end{lemma}

\subsection{Multilinear fractional maximal functions and multiple weights}

%


\begin{definition} [multilinear fractional maximal functions] \label{def.mul-max-frac} \
For all locally integrable functions $\vec{f}=(f_{1},f_{2},\dots,f_{m})$ and $x\in \mathbb{R}^n$, let $0 \le \alpha < mn$.
\begin{enumerate}[leftmargin=2em,label=(\arabic*),itemindent=1.5em]  
\item  The multilinear fractional maximal functions $\mathcal{M}_{\alpha}$ and $\mathcal{M}_{\alpha,r}$ are defined by
\begin{align*}
\mathcal{M}_{\alpha}(\vec{f})(x)   &= \sup\limits_{Q\ni x} |Q|^{\alpha/n}\prod\limits_{j=1}^{m} \frac{1}{|Q|}  \dint_{Q} |f_{j}(y_{j})| d y_{j}  = \sup\limits_{Q\ni x}  \prod\limits_{j=1}^{m} \frac{1}{|Q|^{1-\alpha/(nm)}}  \dint_{Q} |f_{j}(y_{j})| d y_{j},
\\ \intertext{and}
\mathcal{M}_{\alpha, r}(\vec{f})(x)   &= \sup\limits_{Q\ni x} |Q|^{\alpha/n} \prod\limits_{j=1}^{m} \Big(\frac{1}{|Q|}  \dint_{Q} |f_{j}(y_{j})|^{r} d y_{j} \Big)^{1/r}, \text{for}~ r>1,
\end{align*}
where the supremum is taken over all the cubes $Q$ containing $x$.
\item  the  multilinear fractional maximal functions related to Young function $\Phi(t)=t(1+\log^{+}t)$ are defined by
\begin{align*}
\mathcal{M}_{\alpha, L(\log L)}^{i}(\vec{f})(x)  & = \sup\limits_{Q\ni x} |Q|^{\alpha/n}  \|f_{i}\|_{L(\log L) ,Q} \prod\limits_{j=1 \atop j\neq i}^{m} \frac{1}{|Q|}  \dint_{Q} |f_{j}(y_{j})| d y_{j} ,
\\ \intertext{and}
\mathcal{M}_{\alpha, L(\log L)} (\vec{f})(x)   &= \sup\limits_{Q\ni x} |Q|^{\alpha/n}  \prod\limits_{j=1}^{m} \|f_{j}\|_{L(\log L) ,Q} ,
\end{align*}
where the supremum is taken over all the cubes $Q$ containing $x$, and $\|\cdot\|_{L(\log L),Q}$ is the Luxemburg type average defined via
\begin{align*}
\|g\|_{L(\log L),Q} &= \inf \Big\{\lambda>0:  \frac{1}{|Q|} \int_{Q}  \frac{|g(x)|}{\lambda} \log(e+|f|/\lambda)  \mathrm{d}x \le 1 \Big\}.
\end{align*}
\end{enumerate}
\end{definition}

\begin{remark}  \label{rem.multilinear-maximal-fractional}
\begin{enumerate}[leftmargin=2em,label=(\alph*),itemindent=1.5em]  
\item  If  we take $f\equiv 1$ in \eqref{equ:holder-4} with $t=1$, it follows that for every  $\alpha\in [0,mn)$ the
inequality
\begin{align*} \label{inequ:mulmaxfrac-relation}
\mathcal{M}_{\alpha}(\vec{f})(x)  \le C \mathcal{M}_{L(\log L)}^{i} (\vec{f})(x)   &\le C_{1} \mathcal{M}_{\alpha,L(\log L)} (\vec{f})(x).
\end{align*}
\vspace{-2em}
\label{enumerate:mulmaxfrac-relation}
\item  In \cite{bernardis2010composition}, the authors prove that   $M_{\alpha}(M^{k}) \approx M_{\alpha,L(\log L)^{k}}$   with $k \in \mathbb{N}$,  where  $M^{k}$ is the iteration of the Hardy-Littlewood maximal operator $k$ times. In particularly, for $\alpha=0$ and  $k=1$, one have   $M_{L(\log L)} \approx M^{2}=M\circ M$ ~(see also \cite{perez1995endpoint} or \cite{bernardis2006weighted}).
\label{enumerate:maxfrac-relation}
\end{enumerate}
\end{remark}

The following gives the characterization of the multiple-weight class  $A_{\vec{P}}$ and  $A_{\vec{P},q}$, Separately.
\begin{lemma}\label{lem:multiple-weights}
 Let $\vec{w} = (w_{1},\dots ,w_{m})$, $\vec{P} = (p_{1},\dots ,p_{m})$ and $\frac{1}{p}= \frac{1}{p_{1}}+ \cdots + \frac{1}{p_{m}}$ with $1\le p_{j} < \infty~ (j = 1,\dots,m)$.
\begin{enumerate}[leftmargin=2em,label=(\roman*),itemindent=1.5em]  
\item \ $\vec{w}  \in A_{\vec{P}}$  if and only if
\begin{equation*}
\left\{ \begin{aligned}
        w_{j}^{1-p'_{j}} \in A_{mp'_{j}} & \ \   (j = 1,\dots,m)\\ \vspace{-4em}
        u_{\vec{w}} \in A_{mp} &  \ \
                          \end{aligned} \right.,
                          \end{equation*}
where  $w_{j}^{1-p'_{j}}  \in A_{mp'_{j}}$  in the case $p_{j} = 1$   is understood as $w_{j}^{1/m} \in A_{1}~$.
\label{enumerate:multiple-weights}
\item \   Let $0<\alpha \le mn$ and $\frac{1}{q}= \frac{1}{p}-\frac{\alpha}{n}$. Suppose $\vec{w}  \in A_{\vec{P},q}$,  then
\begin{equation*}
\left\{ \begin{aligned}
        w_{j}^{-p'_{j}} \in A_{mp'_{j}} & \ \  (j = 1,\dots,m) \\ \vspace{-4em}
        v_{\vec{w}}^{q}  \in A_{mq} &  \ \
                          \end{aligned} \right.,
                          \end{equation*}
\vspace{-2em}
\label{enumerate:multiple-weights-fract}
\end{enumerate}
\end{lemma}

 The  first part in   \cref{lem:multiple-weights} is   due to Lerner et al. in \cite{lerner2009new} (see Theorem 3.6).  The second of  \cref{lem:multiple-weights}  was   introduced by Moen \cite{moen2009linear} independently (see Theorem 3.4  in  \cite{moen2009linear} or Theorem 2.1 in  \cite{chen2010weighted}).

\begin{remark}\label{rem:multiple-weights}
 Let $\vec{w} = (w_{1},\dots ,w_{m})$, $\vec{P} = (p_{1},\dots ,p_{m})$ and $\frac{1}{p}= \frac{1}{p_{1}}+ \cdots + \frac{1}{p_{m}}$ with $1\le p_{j} < \infty~ (j = 1,\dots,m)$.
\begin{enumerate}[leftmargin=2em,label=(\roman*),itemindent=1.5em]  
\item  When $m = 1$,  $A_{\vec{P},q}$ will be degenerated to the classical $A_{p,q}$ weights,  and  $A_{\vec{P}}$   reduces to the classical $A_{p}$ weights.
\item  (see P.1232 in \cite{lerner2009new}) If $w_{j} \in  A_{p_{j}}~(j=1,\dots,m)$,   then by H\"{o}lder's inequality we have
\begin{align*}
&\;  \Big(\frac{1}{|Q|} \int_{Q}  u_{\vec{w}}(x)  \mathrm{d}x  \Big)^{1/p}  \prod_{j=1}^{m} \Big(\frac{1}{|Q|} \int_{Q}  \big( w_{j} (x)  \big)^{1-p'_{j}}   \mathrm{d}x \Big)^{ 1/p'_{j}}             \\
  &=  \Big(\frac{1}{|Q|} \int_{Q}\prod_{j=1}^{m} \big( w_{j}(x) \big)^{p/p_{j}} \mathrm{d}x  \Big)^{1/p}  \prod_{j=1}^{m} \Big(\frac{1}{|Q|} \int_{Q}  \big( w_{j} (x)  \big)^{1-p'_{j}}   \mathrm{d}x \Big)^{ 1/p'_{j}}          \\
 &\le  \prod_{j=1}^{m} \Big(\frac{1}{|Q|} \int_{Q}   w_{j}(x)  \mathrm{d}x  \Big)^{1/p_{j}}   \Big(\frac{1}{|Q|} \int_{Q}  \big( w_{j} (x)  \big)^{1-p'_{j}}   \mathrm{d}x \Big)^{ 1/p'_{j}}< \infty,
\end{align*}
so we have
\begin{align*}
  \prod_{j=1}^{m} A_{p_{j}}  &\subsetneq     A_{\vec{P}}.
\end{align*}

\item  (see Remark 3.3 and 7.5 in \cite{moen2009weighted}) If $p_{j} \le q_{j}$, $w_{j} \in  A_{p_{j},q_{j}} ~(j=1,\dots,m)$, and $\frac{1}{q}= \frac{1}{q_{1}}+ \cdots + \frac{1}{q_{m}}$,   then by H\"{o}lder's inequality we have
\begin{align*}
&\;  \Big(\frac{1}{|Q|} \int_{Q} \big(v_{\vec{w}}(x)\big)^{q} dx  \Big)^{1/q}  \prod_{j=1}^{m} \Big(\frac{1}{|Q|} \int_{Q}  \big( w_{j} (x)  \big)^{-p'_{j}}   dx \Big)^{ 1/p'_{j}}             \\
  &=  \Big(\frac{1}{|Q|} \int_{Q} \big(\prod_{j=1}^{m}  w_{j}(x) \big)^{q} \mathrm{d}x  \Big)^{1/q}  \prod_{j=1}^{m} \Big(\frac{1}{|Q|} \int_{Q}  \big( w_{j} (x)  \big)^{-p'_{j}}   \mathrm{d}x \Big)^{ 1/p'_{j}}          \\
 &\le  \prod_{j=1}^{m} \Big(\frac{1}{|Q|} \int_{Q} \big(  w_{j}(x) \big)^{q_{j}} \mathrm{d}x  \Big)^{1/q_{j}}   \Big(\frac{1}{|Q|} \int_{Q}  \big( w_{j} (x)  \big)^{-p'_{j}}   \mathrm{d}x \Big)^{ 1/p'_{j}} < \infty,
\end{align*}
and therefore,
\begin{align*}
\bigcup_{q_{1},\cdots,q_{m}}   \prod_{j=1}^{m} A_{p_{j}, q_{j}}  &\subsetneq     A_{\vec{P},q},
\end{align*}
where the union is over all $q_{j} \ge p_{j}$ that satisfy $\frac{1}{q}= \frac{1}{q_{1}}+ \cdots + \frac{1}{q_{m}}$.
\end{enumerate}
\end{remark}


\section{Proof of \cref{thm:dini-multi-fract-endpoint}}

\begin{proof}
Let $B = \|T_{\alpha}\|_{L^{p_1} \times  \cdots\times L^{p_m} \to L^{q,\infty}}$. Fix $\lambda> 0$ and consider functions  $f_{j} \in L^{1}(\mathbb{R}^{n})$ for $1\le j \le m$. Without loss of generality, we may assume that $\|f_{1}\|_{L^{1}(\mathbb{R}^{n})} = \cdots = \|f_{m}\|_{L^{1}(\mathbb{R}^{n})}=1$.  it need to  show that there is a constant $C=C _{m, n,|\omega|_{\dini(1)}} >0$  such  that
\begin{align}\label{equ:weak-norm-estimate}
 |\{x\in \mathbb{R}^{n}: |T_{\alpha}(\vec{f})(x)|> \lambda\}| &\le  C   \Big(\frac{A+ B}{\lambda} \Big)^{\frac{n}{mn-\alpha}},
\end{align}

Set  $\gamma$  be a positive real number to be determined later. Applying the Calder\'{o}n-Zygmund decomposition to each function $f_{j} $
at height $ (\lambda\gamma)^{\frac{n}{mn-\alpha}}$  to obtain ``good'' function $g_{j}$ and  ``bad'' function $b_{j}$ with a sequence of pairwise disjoint cubes $\{Q_{j,k_{j}}\}_{k_{j}=1}^{\infty}$ such that
\begin{align*}
 f_{j} &=   g_{j} + b_{j} = g_{j} + \sum_{k_{j}} b_{j,k_{j}}
\end{align*}
for all $j=1,\dots,m$, where
\begin{enumerate}[leftmargin=2em,label=(P\arabic*),itemindent=1.5em]  
\item   $\supp(b_{j,k_{j}}) \subset Q_{j,k_{j}}$,
\label{enumerate:CZ-decom-1}
\item     $\dint_{\mathbb{R}^{n}} b_{j,k_{j}}(x) \mathrm{d}x =0$,
\label{enumerate:CZ-decom-2}
\item   $\dint_{\mathbb{R}^{n}} |b_{j,k_{j}}(x)| \mathrm{d}x  \le C (\lambda\gamma)^{\frac{n}{mn-\alpha}} |Q_{j,k_{j}}|$,
\label{enumerate:CZ-decom-3}
\item   $ \Big|\bigcup_{k_{j}} Q_{j,k_{j}} \Big| = \sum_{k_{j}} |Q_{j,k_{j}} | \le C (\lambda\gamma)^{\frac{-n}{mn-\alpha}}$,
\label{enumerate:CZ-decom-4}
\item   $ \|b_{j} \|_{L^{1}(\mathbb{R}^{n}) }  \le C $,
\label{enumerate:CZ-decom-5}
\item   $ \|g_{j} \|_{L^{s}(\mathbb{R}^{n}) } \le C (\lambda\gamma)^{\frac{n}{(mn-\alpha)s'}} $ for any $1\le s \le \infty$.  
\label{enumerate:CZ-decom-6}
\end{enumerate}

Let   $c_{j,k_{j}}$ be the center of cube $Q_{j,k_{j}}$ and $l(Q_{j,k_{j}})$ be its side length. Set $Q_{j,k_{j}}^{*} = 8\sqrt{n} Q_{j,k_{j}}$, $\Omega_{j}^{*} =\bigcup_{k_{j}} Q_{j,k_{j}}^{*} ~(j=1,\dots,m)$, and  $\Omega^{*}= \bigcup_{j=1}^{m} \Omega_{j}^{*} $. Now let
\begin{align*}
E_{1} &=  \{x\in \mathbb{R}^{n}: |T_{\alpha}(g_{1},g_{2},\dots, g_{m})(x)|> \lambda/2^{m}\},   \\
E_{2} &=  \{x\in \mathbb{R}^{n}\setminus  \Omega^{*}: |T_{\alpha}(b_{1},g_{2},\dots, g_{m})(x)|> \lambda/2^{m}\},   \\
E_{3} &=  \{x\in \mathbb{R}^{n}\setminus \Omega^{*}: |T_{\alpha}(g_{1},b_{2},\dots, g_{m})(x)|> \lambda/2^{m}\},   \\
\cdots  & \cdots \\
E_{2^{m}} &=  \{x\in \mathbb{R}^{n}\setminus \Omega^{*}: |T_{\alpha}(b_{1},b_{2},\dots, b_{m})(x)|> \lambda/2^{m}\},
\end{align*}
where each  $E_{s} =  \{x\in \mathbb{R}^{n}\setminus \Omega^{*}: |T_{\alpha}(h_{1},h_{2},\dots, h_{m})(x)|> \lambda/2^{m}\}$ with $h_{j} \in \{g_{j},b_{j}\}$ and all the sets  $E_{s}$ are distinct.

It follows from property   \labelcref{enumerate:CZ-decom-4} that
\begin{align*}
|\Omega^{*} | &\le \sum_{j=1}^{m} |\Omega_{j}^{*} |  \le  C  \sum_{j=1}^{m}  \sum_{k_{j}} |Q_{j,k_{j}} | \le C (\lambda\gamma)^{\frac{-n}{mn-\alpha}}.
\end{align*}

 Let us first estimate $E_{1}$ which is the easiest.Note that $ 1/q =  1/p  - \alpha/n$, by the Chebyshev's inequality, the   $L^{p_1}(\mathbb{R}^{n}) \times  \cdots\times L^{p_m}(\mathbb{R}^{n})\to  L^{q,\infty}(\mathbb{R}^{n})$ boundedness of   $T_{\alpha}$ and property \labelcref{enumerate:CZ-decom-6} to obtain
\begin{align*}
|E_{1} | &=  |\{x\in \mathbb{R}^{n}: |T_{\alpha}(g_{1},g_{2},\dots, g_{m})(x)|> \lambda/2^{m}\} |  \\
  &\le \Big(\frac{2^{m}B}{\lambda}\Big)^{q}  \prod_{j=1}^{m} \|g_{j} \|_{L^{p_{j}}(\mathbb{R}^{n}) }^{q}   \\
   &\le \Big(\frac{2^{m}B}{\lambda}\Big)^{q}  \prod_{j=1}^{m}  (\lambda\gamma)^{\frac{nq}{(mn-\alpha)p'_{j}}}   \\
&\le  C \Big(\frac{B}{\lambda}\Big)^{q}    (\lambda\gamma)^{(m-\frac{1}{p})\frac{nq}{(mn-\alpha)}}   \\
&=C  B^{q}  \lambda^{-\frac{n}{mn-\alpha}}  \gamma^{q-\frac{n}{mn-\alpha}}.
\end{align*}

Since
\begin{align}\label{equ:weak-CZD-Es}
\begin{split}
 |\{x\in \mathbb{R}^{n}: |T_{\alpha}(\vec{f})(x)|> \lambda\}|    
 &\le \sum_{s=1}^{2^{m}} |E_{s} | + C |\Omega^{*} | \\
 &\le \sum_{s=2}^{2^{m}} |E_{s} | + C  B^{q}  \lambda^{-\frac{n}{mn-\alpha}}  \gamma^{q-\frac{n}{mn-\alpha}} +C (\lambda\gamma)^{\frac{-n}{mn-\alpha}}.
\end{split}
\end{align}
Thus, it will need  to give the appropriate estimates for each $|E_{s} |$ with $2\le s\le 2^{m}$ to guarantee
the validity of \labelcref{equ:weak-norm-estimate}.

For the sake of clarity, we split the proof into two cases.

 \textbf{Case 1:} when $m=2$,  $\Omega_{1}^{*} =\bigcup_{k_{1}} Q_{1,k_{1}}^{*}$,   $\Omega_{2}^{*} =\bigcup_{k_{2}} Q_{2,k_{2}}^{*}$,   $\Omega^{*}= \Omega_{1}^{*}\bigcup\Omega_{2}^{*}  $, and   $\mathrm{d}\vec{y} =\mathrm{d}y_{1} \mathrm{d}y_{2}$. There leaves only the following three terms to be considered
\begin{align*}
E_{2} &=  \{x\in \mathbb{R}^{n}\setminus  \Omega^{*}: |T_{\alpha}(b_{1},g_{2})(x)|> \lambda/4\},   \\
E_{3} &=  \{x\in \mathbb{R}^{n}\setminus \Omega^{*}: |T_{\alpha}(g_{1},b_{2})(x)|> \lambda/4\},   \\
E_{4} &=  \{x\in \mathbb{R}^{n}\setminus \Omega^{*}: |T_{\alpha}(b_{1},b_{2})(x)|> \lambda/4\}.
\end{align*}

The following, for $s=2,3,4$, will  show that
\begin{align}\label{equ:weak-CZD-Es-2}
|E_{s}| &\le  C A \Big(\frac{\gamma}{\lambda}\Big)^{\frac{n}{2n-\alpha}}\gamma^{\frac{-\alpha}{2n-\alpha}}.
\end{align}

Now, for the term $|E_{2} |$, by Chebyshev's inequality and property \labelcref{enumerate:CZ-decom-2}, we have
\begin{align} \label{equ:weak-CZD-E2}
|E_{2}| &=  |\{x\in \mathbb{R}^{n}\setminus  \Omega^{*}: |T_{\alpha}(b_{1},g_{2})(x)|> \lambda/4\}|    \notag \\
 &\le \dint_{\{x\in \mathbb{R}^{n}\setminus  \Omega^{*}: |T_{\alpha}(b_{1},g_{2})(x)|> \lambda/4\}} \dfrac{|T_{\alpha}(b_{1},g_{2})(x)|}{\lambda/4}   \mathrm{d}x    \notag \\
  &\le  \frac{4}{\lambda}  \sum_{k_{1}}\dint_{  \mathbb{R}^{n}\setminus  \Omega^{*}}  |T_{\alpha}(b_{1,k_{1}},g_{2})(x)| \mathrm{d}x \\
  &\le  \frac{4}{\lambda}  \sum_{k_{1}} \dint_{  \mathbb{R}^{n}\setminus  \Omega^{*}}   \Big|\dint_{  (\mathbb{R}^{n})^{2}} \Big(K_{\alpha}(x,y_{1},y_{2})-K_{\alpha}(x,c_{1,k_{1}},y_{2})  \Big) b_{1,k_{1}}(y_{1}) g_{2}(y_{2})\mathrm{d}\vec{y} \Big| \mathrm{d}x    \notag \\
  &\le  \frac{4  \|g_{2} \|_{L^{\infty}(\mathbb{R}^{n}) } }{\lambda} \sum_{k_{1}} \dint_{Q_{1,k_{1}}}  | b_{1,k_{1}}(y_{1})|   \dint_{  \mathbb{R}^{n}} \dint_{  \mathbb{R}^{n}\setminus  \Omega^{*}}  \Big|K_{\alpha}(x,y_{1},y_{2})-K_{\alpha}(x,c_{1,k_{1}},y_{2})  \Big|  \mathrm{d}x \mathrm{d}\vec{y}.             \notag
\end{align}

For any fixed $k_{1}$,  let $\mathcal {Q}_{1,k_{1}}^{i}= (2^{i+2} \sqrt{n}  Q_{1,k_{1}}) \setminus (2^{i+1} \sqrt{n} Q_{1,k_{1}}) $ with $i=1,2,\dots$. Clearly we have  $\mathbb{R}^{n}\setminus  \Omega^{*} \subset \mathbb{R}^{n}\setminus  Q_{1,k_{1}}^{*} \subset \bigcup_{i=1}^{\infty} \mathcal {Q}_{1,k_{1}}^{i}$.   For any $y_{1} \in Q_{1,k_{1}}$ and  $y_{2} \in \mathbb{R}^{n}$, since $\omega$ is nondecreasing, then it follows from \labelcref{equ:w-CZK-frac-regularity-2} that
\begin{align} \label{equ:weak-CZD-E2-kernel}
&\; \dint_{  \mathbb{R}^{n}\setminus  \Omega^{*}}  \Big|K_{\alpha}(x,y_{1},y_{2})-K_{\alpha}(x,c_{1,k_{1}},y_{2})  \Big|  \mathrm{d}x
\le A \dint_{  \mathbb{R}^{n}\setminus  \Omega^{*}}  \dfrac{ \omega \Big( \frac{|y_{1}-c_{1,k_{1}}|}{ |x-y_{1}|+|x-y_{2}|} \Big) }{\Big (|x-y_{1}| + |x-y_{2}| \Big)^{2n-\alpha}}   \mathrm{d}x     \notag\\
 &\le A  \sum_{i=1}^{\infty}  \dint_{\mathcal {Q}_{1,k_{1}}^{i}}  \dfrac{ \omega \Big( \frac{|y_{1}-c_{1,k_{1}}|}{ |x-y_{1}|+|x-y_{2}|} \Big) }{\Big (|x-y_{1}| + |x-y_{2}| \Big)^{2n-\alpha}}   \mathrm{d}x  \\
&\le A  \sum_{i=1}^{\infty} \omega(2^{-i}) \dint_{\mathcal {Q}_{1,k_{1}}^{i}}  \dfrac{ 1 }{\Big (|x-y_{1}| + |x-y_{2}| \Big)^{2n-\alpha}}   \mathrm{d}x,         \notag
\end{align}
where in the last step we use the facts that, for $x \in \mathcal {Q}_{1,k_{1}}^{i}$  and any $y_{1} \in Q_{1,k_{1}}$,  there has
\begin{align*}
|y_{1}-c_{1,k_{1}}| \le  \frac{1}{2} \sqrt{n} l( Q_{1,k_{1}}) \ & \text{and} \  |x-y_{1}| \ge  2^{i-1} \sqrt{n} l( Q_{1,k_{1}}) .
\end{align*}
Substituting \eqref{equ:weak-CZD-E2-kernel} into \eqref{equ:weak-CZD-E2},
note that the fact (see \eqref{inequ:(m-1)-CZK})  
\begin{align}\label{inequ:(m-1)-CZK-1}
 \dint_{  \mathbb{R}^{n}} \dfrac{1}{\Big (|x-y_{1}| + |x-y_{2}| \Big)^{2n-\alpha}} \mathrm{d}y_{2}  &\le  \dfrac{C }{|x-y_{1}|  ^{n-\alpha}},
\end{align}
 and applying properties \labelcref{enumerate:CZ-decom-3,enumerate:CZ-decom-4,enumerate:CZ-decom-6},  we have
\begin{align*}
|E_{2}|  &\le  \frac{4  \|g_{2} \|_{L^{\infty}(\mathbb{R}^{n}) } }{\lambda} \sum_{k_{1}} \dint_{Q_{1,k_{1}}}  | b_{1,k_{1}}(y_{1})|   \dint_{  \mathbb{R}^{n}} \dint_{  \mathbb{R}^{n}\setminus  \Omega^{*}}  \Big|K_{\alpha}(x,y_{1},y_{2})-K_{\alpha}(x,c_{1,k_{1}},y_{2})  \Big|  \mathrm{d}x \mathrm{d}\vec{y} \\
&\le  \frac{C A}{\lambda} (\lambda\gamma)^{\frac{n}{2n-\alpha}}  \sum_{k_{1}} \sum_{i=1}^{\infty} \omega(2^{-i}) \dint_{Q_{1,k_{1}}}  | b_{1,k_{1}}(y_{1})|   \dint_{\mathcal {Q}_{1,k_{1}}^{i}}  \dfrac{1}{|x-y_{1}|^{n-\alpha}}   \mathrm{d}x \mathrm{d}y_{1} \\
&\le \frac{C A}{\lambda} (\lambda\gamma)^{\frac{n}{2n-\alpha}}  \sum_{k_{1}} \sum_{i=1}^{\infty} \omega(2^{-i})  \dint_{Q_{1,k_{1}}}  | b_{1,k_{1}}(y_{1})|   \dint_{2^{i+2} \sqrt{n}  Q_{1,k_{1}} }  \dfrac{1}{|2^{i-1} \sqrt{n}  Q_{1,k_{1}}|^{1-\alpha/n}}   \mathrm{d}x \mathrm{d}y_{1}    \\
&\le  \frac{C A}{\lambda} (\lambda\gamma)^{\frac{n}{2n-\alpha}}  \sum_{k_{1}} \sum_{i=1}^{\infty} \omega(2^{-i})2^{-i\alpha} |Q_{1,k_{1}}|^{\alpha/n}\dint_{Q_{1,k_{1}}}  | b_{1,k_{1}}(y_{1})|   \mathrm{d}y_{1} \\
&\le   C A \Big(\frac{\gamma}{\lambda}\Big)^{\frac{n}{2n-\alpha}}\gamma^{\frac{-\alpha}{2n-\alpha}}.
\end{align*}

Similarly, we can obtain that      $|E_{3}| \le   C A \big(\frac{\gamma}{\lambda}\big)^{\frac{n}{2n-\alpha}}\gamma^{\frac{-\alpha}{2n-\alpha}}$.

The following estimate   $|E_{4}|$. By Chebyshev's inequality, properties \labelcref{enumerate:CZ-decom-1,enumerate:CZ-decom-2}, we have
\begin{align*}
|E_{4}| &=  |\{x\in \mathbb{R}^{n}\setminus \Omega^{*}: |T_{\alpha}(b_{1},b_{2})(x)|> \lambda/4\}|   \\
  &\le  \frac{4}{\lambda}  \sum_{k_{1},k_{2}}\dint_{  \mathbb{R}^{n}\setminus  \Omega^{*}}  |T_{\alpha}(b_{1,k_{1}},b_{2,k_{2}})(x)| \mathrm{d}x \\
  &\le \frac{4}{\lambda}  \sum_{k_{1},k_{2}} \dint_{  \mathbb{R}^{n}\setminus  \Omega^{*}}   \Big|\dint_{  (\mathbb{R}^{n})^{2}} \Big(K_{\alpha}(x,y_{1},y_{2})-K_{\alpha}(x,c_{1,k_{1}},y_{2})  \Big) b_{1,k_{1}}(y_{1}) b_{2,k_{2}}(y_{2})\mathrm{d}\vec{y} \Big| \mathrm{d}x \\
  &\le \frac{4}{\lambda}  \sum_{k_{1},k_{2}} \dint_{  \mathbb{R}^{n}\setminus  \Omega^{*}}  \dint_{Q_{2,k_{2}}}  \dint_{Q_{1,k_{1}}}  \Big|K_{\alpha}(x,y_{1},y_{2})-K_{\alpha}(x,c_{1,k_{1}},y_{2})  \Big|  |b_{1,k_{1}}(y_{1}) b_{2,k_{2}}(y_{2})| \mathrm{d}\vec{y}   \mathrm{d}x \\
  &\le   \frac{4}{\lambda}  \sum_{k_{1},k_{2}}   \dint_{Q_{2,k_{2}}}  \dint_{Q_{1,k_{1}}} \Big(  \dint_{  \mathbb{R}^{n}\setminus  \Omega^{*}}  \Big|K_{\alpha}(x,y_{1},y_{2})-K_{\alpha}(x,c_{1,k_{1}},y_{2})  \Big|  \mathrm{d}x \Big)  |b_{1,k_{1}}(y_{1}) b_{2,k_{2}}(y_{2})|   \mathrm{d}\vec{y}.
\end{align*}

For any fixed $k_{2}$,  let $\mathcal {Q}_{1,k_{1}}^{i}$ be as above, denote by $\mathcal {Q}_{2,k_{2}}^{h}= (2^{h+2} \sqrt{n}  Q_{2,k_{2}}) \setminus (2^{h+1} \sqrt{n} Q_{2,k_{2}}) $ with $h=1,2,\dots$. Then
\begin{align*}
 \mathbb{R}^{n}\setminus  \Omega^{*} &\subset \mathbb{R}^{n}\setminus  \big(Q_{1,k_{1}}^{*}\bigcup Q_{2,k_{2}}^{*} \big) \subset \bigcup_{h=1}^{\infty} \bigcup_{i=1}^{\infty} \big(\mathcal{Q}_{1,k_{1}}^{i}\bigcap \mathcal {Q}_{2,k_{2}}^{h} \big).
\end{align*}
For any $( y_{1} , y_{2} ) \in Q_{1,k_{1}} \times Q_{2,k_{2}}$, similar to \eqref{equ:weak-CZD-E2-kernel}, we have
\begin{align} \label{equ:weak-CZD-E4-kernel}
 &\; \dint_{  \mathbb{R}^{n}\setminus  \Omega^{*}}  \Big|K_{\alpha}(x,y_{1},y_{2})-K_{\alpha}(x,c_{1,k_{1}},y_{2})  \Big|  \mathrm{d}x
\le A \dint_{  \mathbb{R}^{n}\setminus  \Omega^{*}}  \dfrac{ \omega \Big( \frac{|y_{1}-c_{1,k_{1}}|}{ |x-y_{1}|+|x-y_{2}|} \Big) }{\Big (|x-y_{1}| + |x-y_{2}| \Big)^{2n-\alpha}}   \mathrm{d}x          \notag \\
 &\le A  \sum_{h=1}^{\infty} \sum_{i=1}^{\infty}  \dint_{\mathcal{Q}_{1,k_{1}}^{i}\bigcap \mathcal {Q}_{2,k_{2}}^{h}}  \dfrac{ \omega \Big( \frac{|y_{1}-c_{1,k_{1}}|}{ |x-y_{1}|+|x-y_{2}|} \Big) }{\Big (|x-y_{1}| + |x-y_{2}| \Big)^{2n-\alpha}}   \mathrm{d}x  \\
&\le A  \sum_{h=1}^{\infty}\sum_{i=1}^{\infty} \omega(2^{-i}) \dint_{\mathcal{Q}_{1,k_{1}}^{i}\bigcap \mathcal {Q}_{2,k_{2}}^{h}}  \dfrac{ 1 }{\Big (|x-y_{1}| + |x-y_{2}| \Big)^{2n-\alpha}}   \mathrm{d}x.        \notag
\end{align}
Note that, for any  $x \in\mathcal{Q}_{1,k_{1}}^{i}\bigcap \mathcal {Q}_{2,k_{2}}^{h}$ and $( y_{1} , y_{2} ) \in Q_{1,k_{1}} \times Q_{2,k_{2}}$, there has
\begin{align*}
|x-y_{1}| \approx  2^{i+1} \sqrt{n} l( Q_{1,k_{1}}) \ & \text{and} \  |x-y_{2}| \approx  2^{h+1} \sqrt{n} l( Q_{2,k_{2}}),
\end{align*}
then, for any $( y_{1} , y_{2} ) \in Q_{1,k_{1}} \times Q_{2,k_{2}}$,  the following holds
\begin{align}\label{equ:weak-CZD-E4-kernel-1}
\begin{split}
\dint_{\mathcal{Q}_{1,k_{1}}^{i}\bigcap \mathcal {Q}_{2,k_{2}}^{h}}  \dfrac{ 1 }{\Big (|x-y_{1}| + |x-y_{2}| \Big)^{2n-\alpha}}   \mathrm{d}x  &\approx \dfrac{ \big|\mathcal{Q}_{1,k_{1}}^{i}\bigcap \mathcal {Q}_{2,k_{2}}^{h} \big| }{\Big (2^{i+1} \sqrt{n} l( Q_{1,k_{1}})+  2^{h+1} \sqrt{n} l( Q_{2,k_{2}}) \Big)^{2n-\alpha}}     \\
&:= \mathcal{H}(i,,k_{1};h,,k_{2}).
\end{split}
\end{align}
From estimates \labelcref{equ:weak-CZD-E4-kernel,equ:weak-CZD-E4-kernel-1} to obtain
\begin{align} \label{equ:weak-CZD-E4-kernel-2}
   \dint_{  \mathbb{R}^{n}\setminus  \Omega^{*}}  \Big|K_{\alpha}(x,y_{1},y_{2})-K_{\alpha}(x,c_{1,k_{1}},y_{2})  \Big|  \mathrm{d}x
&\le C A  \sum_{h=1}^{\infty}\sum_{i=1}^{\infty} \omega(2^{-i}) \mathcal{H}(i,,k_{1};h,,k_{2}).
\end{align}
Then, by \eqref{equ:weak-CZD-E4-kernel-2}  and property \labelcref{enumerate:CZ-decom-3} one has
\begin{align*}
|E_{4}| &\le   \frac{4}{\lambda}  \sum_{k_{1},k_{2}}   \dint_{Q_{2,k_{2}}}  \dint_{Q_{1,k_{1}}} \Big(  \dint_{  \mathbb{R}^{n}\setminus  \Omega^{*}}  \Big|K_{\alpha}(x,y_{1},y_{2})-K_{\alpha}(x,c_{1,k_{1}},y_{2})  \Big|  \mathrm{d}x \Big)  |b_{1,k_{1}}(y_{1}) b_{2,k_{2}}(y_{2})|   \mathrm{d}\vec{y} \\
  &\le   \frac{CA}{\lambda}  \sum_{k_{1},k_{2}}   \dint_{Q_{2,k_{2}}}  \dint_{Q_{1,k_{1}}} \Big(   \sum_{h=1}^{\infty}\sum_{i=1}^{\infty} \omega(2^{-i}) \mathcal{H}(i,,k_{1};h,,k_{2}) \Big)  |b_{1,k_{1}}(y_{1}) b_{2,k_{2}}(y_{2})|   \mathrm{d}\vec{y}   \\
  &\le    \frac{CA}{\lambda}  (\lambda\gamma)^{\frac{n}{2n-\alpha}} (\lambda\gamma)^{\frac{n}{2n-\alpha}}    \sum_{i=1}^{\infty} \omega(2^{-i})  \sum_{k_{1},k_{2}}  |Q_{1,k_{1}}| |Q_{2,k_{2}}|  \Big(   \sum_{h=1}^{\infty} \mathcal{H}(i,,k_{1};h,,k_{2}) \Big)     \\
 &\le    \frac{CA}{\lambda}  (\lambda\gamma)^{\frac{2n}{2n-\alpha}}      \sum_{i=1}^{\infty} \omega(2^{-i})  \sum_{k_{1},k_{2}} \dint_{Q_{1,k_{1}}} \dint_{Q_{2,k_{2}}} \Big(   \sum_{h=1}^{\infty}  \mathcal{H}(i,,k_{1};h,,k_{2}) \Big)     \mathrm{d}\vec{y}.
%
\end{align*}
Applying \labelcref{equ:weak-CZD-E4-kernel-1} again and noting that for any fixed $k_{2}$, the sequence $\{ \mathcal {Q}_{2,k_{2}}^{h} \}_{h=1}^{\infty}$  is pairwise disjoint, it follows from property \labelcref{enumerate:CZ-decom-4} and  estimate \eqref{inequ:(m-1)-CZK-1} that
\begin{align*}
|E_{4}| &\le     \frac{CA}{\lambda}  (\lambda\gamma)^{\frac{2n}{2n-\alpha}}   \sum_{i=1}^{\infty} \omega(2^{-i})   \sum_{k_{1},k_{2}} \dint_{Q_{1,k_{1}}} \dint_{Q_{2,k_{2}}} \Big(   \sum_{h=1}^{\infty}  \mathcal{H}(i,,k_{1};h,,k_{2}) \Big)     \mathrm{d}\vec{y} \\
&\le     \frac{CA}{\lambda}  (\lambda\gamma)^{\frac{2n}{2n-\alpha}} \sum_{i=1}^{\infty} \omega(2^{-i})   \sum_{k_{1},k_{2}} \dint_{Q_{1,k_{1}}} \dint_{Q_{2,k_{2}}} \Big(     \dint_{\mathcal{Q}_{1,k_{1}}^{i} }  \dfrac{ 1 }{\Big (|x-y_{1}| + |x-y_{2}| \Big)^{2n-\alpha}}   \mathrm{d}x   \Big)     \mathrm{d}\vec{y} \\
&\le     \frac{CA}{\lambda}  (\lambda\gamma)^{\frac{2n}{2n-\alpha}} \sum_{i=1}^{\infty} \omega(2^{-i})   \sum_{k_{1}} \dint_{Q_{1,k_{1}}} \dint_{\mathcal{Q}_{1,k_{1}}^{i} }     \dfrac{ 1 }{ |x-y_{1}| ^{n-\alpha}}    \mathrm{d}x   \mathrm{d}y_{1}   \\
&\le    \frac{CA}{\lambda}  (\lambda\gamma)^{\frac{2n}{2n-\alpha}} \sum_{i=1}^{\infty} \omega(2^{-i})   \sum_{k_{1}} \dint_{Q_{1,k_{1}}}   \dint_{2^{i+2} \sqrt{n}  Q_{1,k_{1}} }  \dfrac{1}{|2^{i} \sqrt{n}  Q_{1,k_{1}}|^{1-\alpha/n}}   \mathrm{d}x \mathrm{d}y_{1}    \\
&\le    \frac{CA}{\lambda}  (\lambda\gamma)^{\frac{2n}{2n-\alpha}} \sum_{i=1}^{\infty} \omega(2^{-i}) 2^{-i\alpha}   \sum_{k_{1}} |Q_{1,k_{1}}|^{1+\alpha/n}      \\
&\le  C A \Big(\frac{\gamma}{\lambda}\Big)^{\frac{n}{2n-\alpha}}\gamma^{\frac{-\alpha}{2n-\alpha}}.
\end{align*}

It is easy to see that the constants $C$ involved depend only on $m$, $n$ and $|\omega| _{\dini ( 1 )}$. So, \eqref{equ:weak-CZD-Es-2} is proven. Set $\gamma = ( A + B ) ^{-1}$, it follows from \labelcref{equ:weak-CZD-Es,equ:weak-CZD-Es-2} that
\begin{align*}
&\; |\{x\in \mathbb{R}^{n}: |T_{\alpha}(f_{1},f_{2})(x)|> \lambda\}|\le \sum_{s=2}^{2^{2}} |E_{s} |  + C  B^{q}  \lambda^{-\frac{n}{2n-\alpha}}  \gamma^{q-\frac{n}{2n-\alpha}} +C (\lambda\gamma)^{\frac{-n}{2n-\alpha}}     \\
 &\le  C A \Big(\frac{\gamma}{\lambda}\Big)^{\frac{n}{2n-\alpha}}\gamma^{\frac{-\alpha}{2n-\alpha}} +  C  B^{q}  \lambda^{-\frac{n}{2n-\alpha}}  \gamma^{q-\frac{n}{2n-\alpha}} +C (\lambda\gamma)^{\frac{-n}{2n-\alpha}}     \\
 &\le  C \lambda^{-\frac{n}{2n-\alpha}} \Big( A  (A+B)^{\frac{\alpha-n}{2n-\alpha}}  +     B^{q}  (A+B)^{\frac{n}{2n-\alpha}-q} +  (A+B)^{\frac{n}{2n-\alpha}}  \Big)   \\
  &\le  C  \Big(\frac{A+B}{\lambda}\Big)^{\frac{n}{2n-\alpha}} \Big( A  (A+B)^{\frac{\alpha-2n}{2n-\alpha}}  +     B^{q}  (A+B)^{-q} +  1  \Big)   \\
   &\le  C  \Big(\frac{A+B}{\lambda}\Big)^{\frac{n}{2n-\alpha}} \Big( A  (A+B)^{-1}  +     B^{q}  (A+B)^{-q} +  1  \Big)   \\
 &\le  C  \Big(\frac{A+B}{\lambda}\Big)^{\frac{n}{2n-\alpha}},
\end{align*}
which is the desired result. The proof of the case $m = 2$ is completed.

 \textbf{Case 2:} when $m\ge 3$,  we need to estimate $| E_{s} |$ for $2 \le  s \le 2^{m} $. Suppose that for some $1 \le \ell  \le m$, we have $\ell$ bad functions and $m-\ell$ good functions appearing in    $T_{\alpha}(h_{1},h_{2},\dots, h_{m})$ with $h_{j} \in \{g_{j},b_{j}\}$. For matters of simplicity, without loss of generality, we may  assume that the bad functions appear at the entries $1,\dots, \ell$, and denote the corresponding term by $| E _{s}^{(\ell)}|$
 to distinguish it from the other terms.  The following will consider
\begin{align*}
|E_{s}^{(\ell)}| &=  |\{x\in \mathbb{R}^{n}\setminus \Omega^{*}: |T_{\alpha}(b_{1},\dots,b_{\ell}, g_{\ell+1},\dots, g_{m})(x)|> \lambda/2^{m}\}|,
\end{align*}
and the other terms can be estimated similarly. We will show
\begin{align}\label{equ:weak-CZD-Es-m}
|E_{s}^{(\ell)}| &\le  C A \Big(\frac{\gamma}{\lambda}\Big)^{\frac{n}{mn-\alpha}}\gamma^{\frac{(m-2)n-\alpha}{mn-\alpha}}.
\end{align}

Recall that $\supp(b_{1,k_{1}}) \subset Q_{1,k_{1}}$ and $c_{1,k_{1}}$ be the center of cube $Q_{1,k_{1}}$. Denote by $\prod_{r=1}^{\ell} Q_{r,k_{r}} =Q_{1,k_{1}}\times \cdots \times Q_{\ell,k_{\ell}}$ and  $\vec{y}_{*}=(c_{1,k_{1}}, y_{2}, \dots, y_{m})$  for simplicity. Then it follows from properties \labelcref{enumerate:CZ-decom-2,enumerate:CZ-decom-6}  that, for any $x\in  \mathbb{R}^{n}\setminus  \Omega^{*}$,
\begin{align*}
&\; |T_{\alpha}(b_{1},\dots,b_{\ell},g_{\ell+1},\dots, g_{m})(x)|
\le  \sum_{k_{1},\dots,k_{\ell}}  \Big|\dint_{  (\mathbb{R}^{n})^{m}}  K_{\alpha}(x,\vec{y})  \prod_{r=1}^{\ell} b_{r,k_{r}}(y_{r}) \prod_{r=\ell+1}^{m} g_{r}(y_{r})\mathrm{d}\vec{y} \Big|   \\
&\le  \sum_{k_{1},\dots,k_{\ell}}  \dint_{  (\mathbb{R}^{n})^{m}} \Big| K_{\alpha}(x,\vec{y}) - K_{\alpha}(x,\vec{y}_{*})\Big| \prod_{r=1}^{\ell} |b_{r,k_{r}}(y_{r})| \prod_{r=\ell+1}^{m} |g_{r}(y_{r})|\mathrm{d}\vec{y}     \\
&\le C  \prod_{r=\ell+1}^{m} \|g_{r}\|_{L^{\infty}(\mathbb{R}^{n}) }  \sum_{k_{1},\dots,k_{\ell}}  \dint_{  (\mathbb{R}^{n})^{m}} \Big| K_{\alpha}(x,\vec{y}) - K_{\alpha}(x,\vec{y}_{*})\Big| \prod_{r=1}^{\ell} |b_{r,k_{r}}(y_{r})|  \mathrm{d}\vec{y}    \\
&\le C (\lambda\gamma)^{\frac{n(m-\ell)}{mn-\alpha}}  \sum_{k_{1},\dots,k_{\ell}}  \dint_{  (\mathbb{R}^{n})^{m}} \Big| K_{\alpha}(x,\vec{y}) - K_{\alpha}(x,\vec{y}_{*})\Big| \prod_{r=1}^{\ell} |b_{r,k_{r}}(y_{r})|  \mathrm{d}\vec{y}.
\end{align*}
This together with Chebychev's inequality gives
\begin{align*}
|E_{s}^{(\ell)}| &=  |\{x\in \mathbb{R}^{n}\setminus \Omega^{*}: |T_{\alpha}(b_{1},\dots,b_{\ell}, g_{\ell+1},\dots, g_{m})(x)|> \lambda/2^{m}\}|   \\
&\le \frac{2^{m}}{\lambda}  \dint_{\mathbb{R}^{n}\setminus \Omega^{*}} |T_{\alpha}(b_{1},\dots,b_{\ell}, g_{\ell+1},\dots, g_{m})(x)|  \mathrm{d}x    \\
&\le \frac{C}{\lambda} (\lambda\gamma)^{\frac{n(m-\ell)}{mn-\alpha}}  \dint_{\mathbb{R}^{n}\setminus \Omega^{*}} \Big(\sum_{k_{1},\dots,k_{\ell}}  \dint_{  (\mathbb{R}^{n})^{m}} \Big| K_{\alpha}(x,\vec{y}) - K_{\alpha}(x,\vec{y}_{*})\Big| \prod_{r=1}^{\ell} |b_{r,k_{r}}(y_{r})|  \mathrm{d}\vec{y} \Big)  \mathrm{d}x    \\
&\le \frac{C}{\lambda} (\lambda\gamma)^{\frac{n(m-\ell)}{mn-\alpha}} \sum_{k_{1},\dots,k_{\ell}}  \dint_{  (\mathbb{R}^{n})^{m}} \Big(\dint_{\mathbb{R}^{n}\setminus \Omega^{*}} \big| K_{\alpha}(x,\vec{y}) - K_{\alpha}(x,\vec{y}_{*})\big| \mathrm{d}x \Big) \prod_{r=1}^{\ell} |b_{r,k_{r}}(y_{r})|  \mathrm{d}\vec{y} \\
&\le \frac{C}{\lambda} (\lambda\gamma)^{\frac{n(m-\ell)}{mn-\alpha}} \sum_{k_{1},\dots,k_{\ell}}  \dint_{  (\mathbb{R}^{n})^{m-\ell}}  \dint_{\prod_{r=1}^{\ell} Q_{r,k_{r}} }\prod_{r=1}^{\ell} |b_{r,k_{r}}(y_{r})| \Big(\dint_{\mathbb{R}^{n}\setminus \Omega^{*}} \big| K_{\alpha}(x,\vec{y}) - K_{\alpha}(x,\vec{y}_{*})\big| \mathrm{d}x \Big) \mathrm{d}\vec{y}.
\end{align*}
Let $\mathcal {Q}_{r,k_{r}}^{i_{r}}= (2^{i_{r}+2} \sqrt{n}  Q_{r,k_{r}}) \setminus (2^{i_{r}+1} \sqrt{n} Q_{r,k_{r}}) $ for $r=1,2,\dots,\ell$ and $i_{r}=1,2,\dots$. Then
\begin{align*}
 \mathbb{R}^{n}\setminus  \Omega^{*} &\subset \mathbb{R}^{n}\setminus  \big( \bigcup_{r=1}^{\ell} Q_{r,k_{r}}^{*} \big)    \\
 &\subset \bigcup_{i_{1}=1}^{\infty} \cdots\bigcup_{i_{\ell}=1}^{\infty} \big(\mathcal {Q}_{1,k_{1}}^{i_{1}}\bigcap \cdots \bigcap  \mathcal {Q}_{\ell,k_{\ell}}^{i_{\ell}} \big) = \bigcup_{i_{1},\dots,i_{\ell}=1}^{\infty} \big(\bigcap_{r=1}^{\ell}  \mathcal {Q}_{r,k_{r}}^{i_{r}} \big).
\end{align*}
For any $( y_{1},\dots , y_{\ell} ) \in \prod_{r=1}^{\ell} Q_{r,k_{r}}$ and any $( y_{\ell+1},\dots , y_{m} ) \in \mathbb{R}^{n})^{m-\ell}$,  applying \eqref{equ:w-CZK-frac-regularity-2} and the fact that $\omega$ is nondecreasing, similar to \labelcref{equ:weak-CZD-E2-kernel,equ:weak-CZD-E4-kernel}, we have
\begin{align*}
&\; \dint_{  \mathbb{R}^{n}\setminus  \Omega^{*}}  \Big| K_{\alpha}(x,\vec{y}) - K_{\alpha}(x,\vec{y}_{*})  \Big|  \mathrm{d}x
\le A \dint_{  \mathbb{R}^{n}\setminus  \Omega^{*}}  \dfrac{ \omega \Big( \frac{|y_{1}-c_{1,k_{1}}|}{\sum\limits_{j=1}^{m} |x-y_{j}|} \Big) }{\Big (\sum\limits_{j=1}^{m} |x-y_{j}| \Big)^{mn-\alpha}}     \mathrm{d}x     \notag\\
 &\le A  \sum\limits_{i_{1},\dots,i_{\ell}=1}^{\infty}  \dint_{\bigcap_{r=1}^{\ell}  \mathcal {Q}_{r,k_{r}}^{i_{r}}}  \dfrac{ \omega \Big( \frac{|y_{1}-c_{1,k_{1}}|}{ |x-y_{1}|} \Big) }{\Big (\sum\limits_{j=1}^{m} |x-y_{j}| \Big)^{mn-\alpha}}   \mathrm{d}x  \\
&\le A  \sum\limits_{i_{1},\dots,i_{\ell}=1}^{\infty} \omega(2^{-i_{1}}) \dint_{\bigcap_{r=1}^{\ell}  \mathcal {Q}_{r,k_{r}}^{i_{r}}}  \dfrac{ 1 }{\Big (\sum\limits_{j=1}^{m} |x-y_{j}|  \Big)^{mn-\alpha}}   \mathrm{d}x,         \notag
\end{align*}
Then, 
\begin{align*}
|E_{s}^{(\ell)}|  
&\le \frac{CA}{\lambda} (\lambda\gamma)^{\frac{n(m-\ell)}{mn-\alpha}} \sum_{k_{1},\dots,k_{\ell}} \sum\limits_{i_{1},\dots,i_{\ell}=1}^{\infty} \omega(2^{-i_{1}}) \dint_{  (\mathbb{R}^{n})^{m-\ell}}  \dint_{\prod_{r=1}^{\ell} Q_{r,k_{r}} }\prod_{r=1}^{\ell} |b_{r,k_{r}}(y_{r})|    \\
&\: \hspace{2em} \times \Big(\dint_{\bigcap_{r=1}^{\ell}  \mathcal {Q}_{r,k_{r}}^{i_{r}}}  \dfrac{ 1 }{\Big (\sum\limits_{j=1}^{m} |x-y_{j}|  \Big)^{mn-\alpha}}   \mathrm{d}x\Big) \mathrm{d}\vec{y}   \\
&\le \frac{CA}{\lambda} (\lambda\gamma)^{\frac{n(m-\ell)}{mn-\alpha}} \sum_{k_{1},\dots,k_{\ell}} \sum\limits_{i_{1},\dots,i_{\ell}=1}^{\infty} \omega(2^{-i_{1}})   \dint_{\prod_{r=1}^{\ell} Q_{r,k_{r}} }\prod_{r=1}^{\ell} |b_{r,k_{r}}(y_{r})|    \\
&\: \hspace{2em} \times \Big(\dint_{\bigcap_{r=1}^{\ell}  \mathcal {Q}_{r,k_{r}}^{i_{r}}} \Big( \dint_{  (\mathbb{R}^{n})^{m-\ell}} \dfrac{ 1 }{\Big (\sum\limits_{j=1}^{m} |x-y_{j}|  \Big)^{mn-\alpha}}  \mathrm{d}y_{\ell+1}\cdots \mathrm{d}y_{m} \Big) \mathrm{d}x\Big) \mathrm{d} y_{1} \cdots\mathrm{d} y_{\ell}   \\
&\le \frac{CA}{\lambda} (\lambda\gamma)^{\frac{n(m-\ell)}{mn-\alpha}} \sum_{k_{1},\dots,k_{\ell}} \sum\limits_{i_{1},\dots,i_{\ell}=1}^{\infty} \omega(2^{-i_{1}})   \dint_{\prod_{r=1}^{\ell} Q_{r,k_{r}} }\prod_{r=1}^{\ell} |b_{r,k_{r}}(y_{r})|    \\
&\: \hspace{2em} \times \Big(\dint_{\bigcap_{r=1}^{\ell}  \mathcal {Q}_{r,k_{r}}^{i_{r}}}   \dfrac{ 1 }{\Big (\sum\limits_{j=1}^{\ell} |x-y_{j}|  \Big)^{n\ell-\alpha}}    \mathrm{d}x\Big) \mathrm{d} y_{1} \cdots\mathrm{d} y_{\ell}.
\end{align*}
On the other hand, similar to \eqref{equ:weak-CZD-E4-kernel-1}, for any  $( y_{1},\dots , y_{\ell} ) \in \prod_{r=1}^{\ell} Q_{r,k_{r}}$, there has
\begin{equation}\label{equ:weak-CZD-Em-kernel-1}
\dint_{\bigcap_{r=1}^{\ell}  \mathcal {Q}_{r,k_{r}}^{i_{r}}}  \dfrac{ 1 }{\Big (\sum\limits_{r=1}^{\ell} |x-y_{r}|  \Big)^{n\ell-\alpha}}    \mathrm{d}x  \approx \dfrac{ \big|\bigcap_{r=1}^{\ell}  \mathcal {Q}_{r,k_{r}}^{i_{r}} \big| }{\Big (\sum\limits_{r=1}^{\ell} 2^{i_{r}+1} \sqrt{n} l( Q_{r,k_{r}})   \Big)^{n\ell-\alpha}}.
\end{equation}
Then by \eqref{equ:weak-CZD-Em-kernel-1} and the property \labelcref{enumerate:CZ-decom-3}, we have
\begin{align*}
|E_{s}^{(\ell)}| &\le \frac{CA}{\lambda} (\lambda\gamma)^{\frac{n(m-\ell)}{mn-\alpha}} \sum_{k_{1},\dots,k_{\ell}} \sum\limits_{i_{1},\dots,i_{\ell}=1}^{\infty} \omega(2^{-i_{1}})   \dint_{\prod_{r=1}^{\ell} Q_{r,k_{r}} }\prod_{r=1}^{\ell} |b_{r,k_{r}}(y_{r})|    \\
&\: \hspace{2em} \times \Big(\dint_{\bigcap_{r=1}^{\ell}  \mathcal {Q}_{r,k_{r}}^{i_{r}}}   \dfrac{ 1 }{\Big (\sum\limits_{j=1}^{m} |x-y_{r}|  \Big)^{n\ell-\alpha}}    \mathrm{d}x\Big) \mathrm{d} y_{1} \cdots\mathrm{d} y_{\ell}   \\
 &\le \frac{CA}{\lambda} (\lambda\gamma)^{\frac{n(m-\ell)}{mn-\alpha}} \sum_{k_{1},\dots,k_{\ell}} \sum\limits_{i_{1},\dots,i_{\ell}=1}^{\infty} \omega(2^{-i_{1}})   \dint_{\prod_{r=1}^{\ell} Q_{r,k_{r}} }\prod_{r=1}^{\ell} |b_{r,k_{r}}(y_{r})|    \\
&\: \hspace{2em} \times  \dfrac{ \big|\bigcap_{r=1}^{\ell}  \mathcal {Q}_{r,k_{r}}^{i_{r}} \big| }{\Big (\sum\limits_{r=1}^{\ell} 2^{i_{r}+1} \sqrt{n} l( Q_{r,k_{r}})   \Big)^{n\ell-\alpha}}  \mathrm{d} y_{1} \cdots\mathrm{d} y_{\ell}   \\
 &\le \frac{CA}{\lambda} (\lambda\gamma)^{\frac{nm}{mn-\alpha}} \sum_{k_{1},\dots,k_{\ell}} \sum\limits_{i_{1},\dots,i_{\ell}=1}^{\infty} \omega(2^{-i_{1}})   \prod_{r=1}^{\ell} |Q_{r,k_{r}}|    \dfrac{ \big|\bigcap_{r=1}^{\ell}  \mathcal {Q}_{r,k_{r}}^{i_{r}} \big| }{\Big (\sum\limits_{r=1}^{\ell} 2^{i_{r}+1} \sqrt{n} l( Q_{r,k_{r}})   \Big)^{n\ell-\alpha}}    \\
 &\le \frac{CA}{\lambda} (\lambda\gamma)^{\frac{nm}{mn-\alpha}} \sum_{k_{1},\dots,k_{\ell}} \sum\limits_{i_{1},\dots,i_{\ell}=1}^{\infty} \omega(2^{-i_{1}})   \dint_{\prod_{r=1}^{\ell} Q_{r,k_{r}}}    \dfrac{ \big|\bigcap_{r=1}^{\ell}  \mathcal {Q}_{r,k_{r}}^{i_{r}} \big| }{\Big (\sum\limits_{r=1}^{\ell} 2^{i_{r}+1} \sqrt{n} l( Q_{r,k_{r}})   \Big)^{n\ell-\alpha}}  \mathrm{d} y_{1} \cdots\mathrm{d} y_{\ell}.
\end{align*}
Applying \eqref{equ:weak-CZD-Em-kernel-1} again, we can see that $|E_{s}^{(\ell)}|$ is dominated by
\begin{align*}
&\: \frac{A}{\lambda} (\lambda\gamma)^{\frac{nm}{mn-\alpha}} \sum_{k_{1},\dots,k_{\ell}} \sum\limits_{i_{1},\dots,i_{\ell}=1}^{\infty} \omega(2^{-i_{1}})   \dint_{\prod_{r=1}^{\ell} Q_{r,k_{r}}}    \dfrac{ \big|\bigcap_{r=1}^{\ell}  \mathcal {Q}_{r,k_{r}}^{i_{r}} \big| }{\Big (\sum\limits_{r=1}^{\ell} 2^{i_{r}+1} \sqrt{n} l( Q_{r,k_{r}})   \Big)^{n\ell-\alpha}}  \mathrm{d} y_{1} \cdots\mathrm{d} y_{\ell}   \\
&=\frac{A}{\lambda} (\lambda\gamma)^{\frac{nm}{mn-\alpha}} \sum_{k_{1},\dots,k_{\ell}} \sum\limits_{i_{1},\dots,i_{\ell}=1}^{\infty} \omega(2^{-i_{1}})   \dint_{\prod_{r=1}^{\ell} Q_{r,k_{r}}}   \bigg(\dint_{\bigcap_{r=1}^{\ell}  \mathcal {Q}_{r,k_{r}}^{i_{r}}}  \dfrac{ 1 }{\Big (\sum\limits_{r=1}^{\ell} |x-y_{r}|  \Big)^{n\ell-\alpha}}    \mathrm{d}x \bigg) \mathrm{d} y_{1} \cdots\mathrm{d} y_{\ell}            \\
&=\frac{A}{\lambda} (\lambda\gamma)^{\frac{nm}{mn-\alpha}} \sum\limits_{i_{1},\dots,i_{\ell}=1}^{\infty} \omega(2^{-i_{1}})   \sum_{k_{1}} \dint_{  Q_{1,k_{1}}}   \bigg(\dint_{\bigcap_{r=1}^{\ell}  \mathcal {Q}_{r,k_{r}}^{i_{r}}}  \bigg( \sum_{k_{2},\dots,k_{\ell}}\dint_{\prod_{r=2}^{\ell} Q_{r,k_{r}}}  \dfrac{ \mathrm{d} y_{2} \cdots\mathrm{d} y_{\ell}  }{\big (\sum\limits_{r=1}^{\ell} |x-y_{r}|  \big)^{n\ell-\alpha}}  \bigg)  \mathrm{d}x \bigg) \mathrm{d} y_{1}.
\end{align*}

Since for any fixed $r$, the family $\{ Q_{r,k_{r}}  \}_{k_{r}=1}^{\infty}$   is a sequence of pairwise disjoint cubes, then for any $x \in\bigcap_{r=1}^{\ell}  \mathcal {Q}_{r,k_{r}}^{i_{r}}$ and $ y_{1} \in Q_{1,k_{1}}$, using estimate \eqref{inequ:(m-1)-CZK-1}, there has
\begin{align*}
&\; \sum_{k_{2},\dots,k_{\ell}}\dint_{\prod_{r=2}^{\ell} Q_{r,k_{r}}}  \dfrac{ 1  }{\big (\sum\limits_{r=1}^{\ell} |x-y_{r}|  \big)^{n\ell-\alpha}} \mathrm{d} y_{2} \cdots\mathrm{d} y_{\ell}       \\
&\le \sum_{k_{2},\dots,k_{\ell-1}}\dint_{\prod_{r=2}^{\ell-1} Q_{r,k_{r}}}   \Big(  \dint_{ \mathbb{R}^{n}}  \dfrac{ 1  }{\big (\sum\limits_{r=1}^{\ell} |x-y_{r}|  \big)^{n\ell-\alpha}}\mathrm{d} y_{\ell} \Big) \mathrm{d} y_{2} \cdots\mathrm{d} y_{\ell-1}   \\
&\le C\sum_{k_{2},\dots,k_{\ell-1}}\dint_{\prod_{r=2}^{\ell-1} Q_{r,k_{r}}}      \dfrac{ 1  }{\big (\sum\limits_{r=1}^{\ell-1} |x-y_{r}|  \big)^{n(\ell-1)-\alpha}}  \mathrm{d} y_{2} \cdots\mathrm{d} y_{\ell-1}   \\
&\le \cdots \le  C\sum_{k_{2}}\dint_{  Q_{2,k_{2}}}      \dfrac{ 1  }{\big (  |x-y_{1}|+ |x-y_{2}| \big)^{2n-\alpha}}  \mathrm{d} y_{2}    \\
&\le    C \dint_{\mathbb{R}^{n}}      \dfrac{ 1  }{\big (  |x-y_{1}|+ |x-y_{2}| \big)^{2n-\alpha}}  \mathrm{d} y_{2}    \\
&\le         \dfrac{ C}{  |x-y_{1}| ^{n-\alpha}}.
\end{align*}
Therefore,
\begin{align}\label{equ:weak-CZD-Es-estimate}
|E_{s}^{(\ell)}|  
 &\le \frac{CA}{\lambda} (\lambda\gamma)^{\frac{nm}{mn-\alpha}} \sum\limits_{i_{1},\dots,i_{\ell}=1}^{\infty} \omega(2^{-i_{1}}) \\
 &\; \hspace{2em}\times  \sum_{k_{1}} \dint_{  Q_{1,k_{1}}}   \bigg(\dint_{\bigcap_{r=1}^{\ell}  \mathcal {Q}_{r,k_{r}}^{i_{r}}}  \bigg( \sum_{k_{2},\dots,k_{\ell}}\dint_{\prod_{r=2}^{\ell} Q_{r,k_{r}}}  \dfrac{ \mathrm{d} y_{2} \cdots\mathrm{d} y_{\ell}  }{\big (\sum\limits_{r=1}^{\ell} |x-y_{r}|  \big)^{n\ell-\alpha}}  \bigg)  \mathrm{d}x \bigg) \mathrm{d} y_{1}    \notag\\
 &\le \frac{CA}{\lambda} (\lambda\gamma)^{\frac{nm}{mn-\alpha}} \sum\limits_{i_{1},\dots,i_{\ell}=1}^{\infty} \omega(2^{-i_{1}})   \sum_{k_{1}} \dint_{  Q_{1,k_{1}}}   \bigg(\dint_{\bigcap_{r=1}^{\ell}  \mathcal {Q}_{r,k_{r}}^{i_{r}}}   \dfrac{ 1}{  |x-y_{1}| ^{n-\alpha}}     \mathrm{d}x \bigg) \mathrm{d} y_{1}   \notag \\
 &\le \frac{CA}{\lambda} (\lambda\gamma)^{\frac{nm}{mn-\alpha}} \sum_{i_{1}=1}^{\infty} \omega(2^{-i_{1}})   \sum_{k_{1}} \dint_{  Q_{1,k_{1}}}   \bigg(\sum_{i_{2},\dots,i_{\ell}=1}^{\infty} \dint_{\bigcap_{r=1}^{\ell}  \mathcal {Q}_{r,k_{r}}^{i_{r}}}   \dfrac{ 1}{  |x-y_{1}| ^{n-\alpha}}     \mathrm{d}x \bigg) \mathrm{d} y_{1}.             \notag
\end{align}

On the other hand, noting that for any fixed $r$  the sequence $\{\mathcal {Q}_{r,k_{r}}^{i_{r}}\}_{i_{r}=1}^{\infty}$ is also pairwise disjoint, then for any $y_{1} \in Q_{1,k_{1}}$, there has
\begin{align*}
\sum_{i_{2},\dots,i_{\ell}=1}^{\infty} \dint_{\bigcap_{r=1}^{\ell}  \mathcal {Q}_{r,k_{r}}^{i_{r}}}   \dfrac{ 1}{  |x-y_{1}| ^{n-\alpha}}     \mathrm{d}x 
 &=  \sum_{i_{2},\dots,i_{\ell-1}=1}^{\infty} \Big( \dint_{\big(\bigcap_{r=1}^{\ell-1}  \mathcal {Q}_{r,k_{r}}^{i_{r}}\big) \bigcap \big(\bigcup_{i_{\ell}=1}^{\infty}  \mathcal {Q}_{\ell,k_{\ell}}^{i_{\ell}}\big) }   \dfrac{ 1}{  |x-y_{1}| ^{n-\alpha}}     \mathrm{d}x \Big)  \\
  &\le  \sum_{i_{2},\dots,i_{\ell-1}=1}^{\infty} \Big( \dint_{ \bigcap_{r=1}^{\ell-1}  \mathcal {Q}_{r,k_{r}}^{i_{r}}  }   \dfrac{ 1}{  |x-y_{1}| ^{n-\alpha}}     \mathrm{d}x \Big)  \\
 &\le \cdots \le  \sum_{i_{2}=1}^{\infty} \Big( \dint_{  \mathcal {Q}_{1,k_{1}}^{i_{1}} \bigcap \mathcal {Q}_{2,k_{2}}^{i_{2}} }   \dfrac{ 1}{  |x-y_{1}| ^{n-\alpha}}     \mathrm{d}x \Big)  \\
 &\le    \dint_{  \mathcal {Q}_{1,k_{1}}^{i_{1}}  }   \dfrac{ 1}{  |x-y_{1}| ^{n-\alpha}}     \mathrm{d}x.
\end{align*}
Substituting the above estimate into \eqref{equ:weak-CZD-Es-estimate} and applying property \labelcref{enumerate:CZ-decom-4}, we have
\begin{align*}
|E_{s}^{(\ell)}| &\le  \frac{CA}{\lambda} (\lambda\gamma)^{\frac{nm}{mn-\alpha}} \sum_{i_{1}=1}^{\infty} \omega(2^{-i_{1}})   \sum_{k_{1}} \dint_{  Q_{1,k_{1}}}   \bigg(\sum_{i_{2},\dots,i_{\ell}=1}^{\infty} \dint_{\bigcap_{r=1}^{\ell}  \mathcal {Q}_{r,k_{r}}^{i_{r}}}   \dfrac{ 1}{  |x-y_{1}| ^{n-\alpha}}     \mathrm{d}x \bigg) \mathrm{d} y_{1}    \\
&\le  \frac{CA}{\lambda} (\lambda\gamma)^{\frac{nm}{mn-\alpha}} \sum_{i_{1}=1}^{\infty} \omega(2^{-i_{1}})   \sum_{k_{1}} \dint_{  Q_{1,k_{1}}}   \bigg(\dint_{ 2^{i_{1}+2}\sqrt{n}  Q_{1,k_{1}}   }   \dfrac{ 1}{  | 2^{i_{1}}\sqrt{n}  Q_{1,k_{1}} | ^{1-\alpha/n}}     \mathrm{d}x \bigg) \mathrm{d} y_{1}    \\
&\le    \frac{CA}{\lambda} (\lambda\gamma)^{\frac{nm}{mn-\alpha}}  \sum_{i=1}^{\infty} \omega(2^{-i_{1}}) 2^{-i_{1}\alpha}   \sum_{k_{1}} |Q_{1,k_{1}}|^{1+\alpha/n}      \\
&\le    \frac{CA}{\lambda} (\lambda\gamma)^{\frac{nm}{mn-\alpha}} (\lambda\gamma)^{\frac{-n(1+\alpha/n)}{mn-\alpha}} \sum_{i=1}^{\infty} \omega(2^{-i_{1}}) 2^{-i_{1}\alpha}      \\
&\le  C A \Big(\frac{\gamma}{\lambda}\Big)^{\frac{n}{mn-\alpha}}\gamma^{\frac{(m-2)n-\alpha}{mn-\alpha}}.
\end{align*}
This shows that \eqref{equ:weak-CZD-Es-m} holds.

Now, we have proved that each $|E_{s} |$ satisfies $|E_{s} | \le  C A \big(\frac{\gamma}{\lambda}\big)^{\frac{n}{mn-\alpha}}\gamma^{\frac{(m-2)n-\alpha}{mn-\alpha}}$. So, by \labelcref{equ:weak-CZD-Es,equ:weak-CZD-Es-m},
set $\gamma = ( A + B ) ^{-1}$, we have
\begin{align*}
&\; |\{x\in \mathbb{R}^{n}: |T_{\alpha}(\vec{f})(x)|> \lambda\}|  
  \le \sum_{s=2}^{2^{m}} |E_{s} | + C  B^{q}  \lambda^{-\frac{n}{mn-\alpha}}  \gamma^{q-\frac{n}{mn-\alpha}} +C (\lambda\gamma)^{\frac{-n}{mn-\alpha}}  \\
 &\le  C A \Big(\frac{\gamma}{\lambda}\Big)^{\frac{n}{mn-\alpha}}\gamma^{\frac{(m-2)n-\alpha}{mn-\alpha}} + C  B^{q}  \lambda^{-\frac{n}{mn-\alpha}}  \gamma^{q-\frac{n}{mn-\alpha}} +C (\lambda\gamma)^{\frac{-n}{mn-\alpha}}            \\
 &\le  C\Big(\frac{1}{\lambda}\Big)^{\frac{n}{mn-\alpha}} \big( A \gamma^{\frac{(m-1)n-\alpha}{mn-\alpha}} +   B^{q}  \gamma^{q-\frac{n}{mn-\alpha}} +  \gamma^{\frac{-n}{mn-\alpha}} \big)          \\
 &=  C\Big(\frac{A+B}{\lambda}\Big)^{\frac{n}{mn-\alpha}}.
\end{align*}

 The proof of  \cref{thm:dini-multi-fract-endpoint} is finished.
\end{proof}

\section{Proof of \cref{thm:dini-multi-fract-czo-weight}}

In order  to prove \cref{thm:dini-multi-fract-czo-weight}, we first establish the following pointwise estimates on the $\delta$-sharp maximal  function
 acting on  $T_{_{\alpha}}(\vec{f})$ controlled in terms of the multilinear maximal function   $\mathcal{M}_{\alpha}$

\begin{theorem} \label{thm.dini-CZ-sharp-max-fract}
Let $m\ge 2$, $0<\alpha<mn$, $T_{\alpha}$ be an $m$-linear  $\omega_{\alpha}$-CZO with   $\omega \in \dini(1)$.
Assume that  $0<\delta<    1$ and $0<\delta< n/(mn-\alpha)$. Then for all $\vec{f}$ in any product space $ L^{p_1}(\mathbb{R}^{n}) \times \cdots\times L^{p_m}(\mathbb{R}^{n}) $ with $1\le p_{j}< \infty~(j=1,2,\dots,m)$, there exists a positive constant $C$ such that
\begin{align*}
  M_{\delta}^{\sharp}(T_{_{\alpha}} (\vec{f})) (x)   &\le C   \mathcal{M}_{\alpha}(\vec{f})(x).
\end{align*}
\end{theorem}

\begin{proof}
 For a fixed point $x$ and a cube $Q  \ni x$. Due to the fact $\big| | a|^{\gamma} - | c|^{\gamma} \big| \le | a-c|^{\gamma}$ for $0<\gamma<1$,  it suffices to prove that, for $0<\delta< \min\{1,n/(mn-\alpha)\}$, there exists a positive constant $C$ such that
\begin{equation*}
   \left(\dfrac{1}{|Q|}\dint_{Q} \Big| T_{_{\alpha}}(\vec{f})(z) -c \Big|^{\delta} \mathrm{d}z \right)^{1/\delta} \le C    \mathcal{M}_{\alpha}(\vec{f})(x),
\end{equation*}
where the constant  $c$ is to be determined later.

For each $j$, we decompose $f_{j}=f_{j}^{0} + f_{j}^{\infty}$ with $f_{j}^{0}=f_{j} \chi_{Q^{*}} $ and $Q^{*} = 8\sqrt{n} Q$. 
Then
\begin{align*}
 \prod_{j=1}^{m} f_{j}(y_{j})   &=  \prod_{j=1}^{m} \Big( f_{j}^{0}(y_{j}) + f_{j}^{\infty}(y_{j}) \Big)
   =   \sum_{\rho_{1},\dots,\rho_{m}\in \{0,\infty\}}  f_{1}^{\rho_{1}}(y_{1}) \cdots  f_{m}^{\rho_{m}} (y_{m}) \\
  &=  \prod_{j=1}^{m} f_{j}^{0}(y_{j}) +  \sum_{(\rho_{1},\dots,\rho_{m})\in \rho} f_{1}^{\rho_{1}}(y_{1}) \cdots  f_{m}^{\rho_{m}} (y_{m}),
\end{align*}
where $\rho=\{ (\rho_{1},\dots,\rho_{m}): \hbox{there is at least one}~ \rho_{j}\neq 0\}$.  It is easy to see that
\begin{align*}
 T_{_{\alpha}}(\vec{f})(z)  &= T_{\alpha}(f_{1}^{0},\dots,  f_{m}^{0})(z) + \sum_{(\rho_{1},\dots,\rho_{m})\in \rho} T_{\alpha}(f_{1}^{\rho_{1}},\dots,  f_{m}^{\rho_{m}})(z) .
\end{align*}
Furthermore, we have
\begin{align*}
   \left(\dfrac{1}{|Q|}\dint_{Q} \Big| T_{_{\alpha}}(\vec{f})(z) -c \Big|^{\delta} \mathrm{d}z \right)^{1/\delta} &\le  \left(\dfrac{1}{|Q|}\dint_{Q} \Big| T_{\alpha}(f_{1}^{0},\dots,  f_{m}^{0})(z) \Big|^{\delta} \mathrm{d}z \right)^{1/\delta}      \\
&\; \hspace{1em}  +  \left(\dfrac{1}{|Q|}\dint_{Q} \Big| \sum_{(\rho_{1},\dots,\rho_{m})\in \rho} T_{\alpha}(f_{1}^{\rho_{1}},\dots,  f_{m}^{\rho_{m}})(z) -c \Big|^{\delta} \mathrm{d}z \right)^{1/\delta}    \\
& :=  I_{1}+I_{2}.
\end{align*}

Let us first estimate $I_{1}$. Since $T_{\alpha}$ is an $m$-linear  $\omega_{\alpha}$-CZO with   $\omega \in \dini(1)$, then it follows from  \cref{thm:dini-multi-fract-endpoint} that $T_{\alpha}$ maps  $L^{1}(\mathbb{R}^{n}) \times  \cdots\times L^{1}(\mathbb{R}^{n}) $ into  $L^{\frac{n}{mn-\alpha},\infty}(\mathbb{R}^{n})$.  Applying Kolmogorov's inequality (see \cref{lem:kolmogorov}) with $p=\delta$ and $q=\frac{n}{mn-\alpha}$, we have
\begin{align*}
 I_{1}  &=  \left(\dfrac{1}{|Q|}\dint_{Q} \Big| T_{\alpha}(f_{1}^{0},\dots,  f_{m}^{0})(z) \Big|^{\delta} dz \right)^{1/\delta}      \\
&\le C |Q|^{-\frac{mn-\alpha}{n}} \|T_{\alpha}(f_{1}^{0},\dots,  f_{m}^{0})\|_{L^{\frac{n}{mn-\alpha},\infty}(\mathbb{R}^{n})}   \\
&\le C  \prod_{j=1}^{m} \frac{1}{|Q^{*}|^{1-\frac{\alpha}{mn}}} \dint_{Q^{*}} |f_{j}(z)| \mathrm{d} z    \\
&\le C    \mathcal{M}_{\alpha}(\vec{f})(x).
\end{align*}

To estimate the remaining terms in $I_{2}$, we choose
\begin{align*}
c  &=  \sum\limits_{(\rho_{1},\dots,\rho_{m})\in \rho}   T_{\alpha}(f_{1}^{\rho_{1}},\dots, f_{m}^{\rho_{m}})(x),
\end{align*}
and it suffices to show that, for any $z \in Q$, the following estimates hold
\begin{align} \label{inequ:dini-frac-max}
\sum_{(\rho_{1},\dots,\rho_{m})\in \rho} \Big| T_{\alpha}(f_{1}^{\rho_{1}},\dots,  f_{m}^{\rho_{m}})(z)      -    T_{\alpha}(f_{1}^{\rho_{1}},\dots, f_{m}^{\rho_{m}})(x) \Big|  &\le C    \mathcal{M}_{\alpha}(\vec{f})(x).
\end{align}

we consider first the case when $\rho_{1}=\cdots =\rho_{m}=\infty$. For any $x, z\in Q$, there has
\begin{align*}
&  \Big|T_{\alpha}(f_{1}^{\infty},\dots, f_{m}^{\infty})(z)   -  T_{\alpha}(f_{1}^{\infty},\dots, f_{m}^{\infty})(x)    \Big|  \\
&\le    \int_{(\mathbb{R}^{n}\setminus  Q^{*})^{m}}   |K_{\alpha}(z,\vec{y})-K_{\alpha}(x,\vec{y})| \prod_{j=1}^{m}| f_{j}^{\infty}(y_{j})|  d\vec{y} \\
&\le   \sum_{k=1}^{\infty} \int_{(\mathcal {Q}_{k})^{m}}   |K_{\alpha}(z,\vec{y})-K_{\alpha}(x,\vec{y})| \prod_{j=1}^{m}| f_{j}^{\infty}(y_{j})|  d\vec{y},
\end{align*}
where $\mathcal {Q}_{k}= ( 2^{k+3}\sqrt{n}Q) \setminus (2^{k+2}\sqrt{n}Q) $ for $k=1,2,\dots$. Note that, for $x,z\in Q$ and any $(y_{1},\dots,y_{m}) \in (\mathcal {Q}_{k})^{m}$, there has
\begin{align*}
2^{k}\sqrt{n}l(Q)\le |z-y_{j}| \ & ~\hbox{and}\ |z-x|\le \sqrt{n} l(Q),
\end{align*}
and recalling that $\omega$  is   nondecreasing, and applying \eqref{equ:w-CZK-frac-regularity-1}, we have
\begin{align} \label{equ:w-CZK-frac-regularity-1-1}
|K_{\alpha}(z,\vec{y})-K_{\alpha}(x,\vec{y})| &\le \dfrac{A }{\Big (\sum\limits_{j=1}^{m}|z-y_{j}| \Big)^{mn-\alpha}} \omega \bigg( \frac{|z-x|}{ \sum\limits_{j=1}^{m}|z-y_{j}|} \bigg)
\le \dfrac{C \omega (2^{-k})}{|2^{k}\sqrt{n} Q |^{m-\alpha/n}}.
\end{align}
Then
\begin{align*}
&\;  \Big|T_{\alpha}(f_{1}^{\infty},\dots, f_{m}^{\infty})(z)   -  T_{\alpha}(f_{1}^{\infty},\dots, f_{m}^{\infty})(x)    \Big|   \\
 &\le   \sum_{k=1}^{\infty} \int_{(\mathcal {Q}_{k})^{m}}   |K_{\alpha}(z,\vec{y})-K_{\alpha}(x,\vec{y})| \prod_{j=1}^{m}| f_{j}^{\infty}(y_{j})|  d\vec{y} \\
&\le   C \sum_{k=1}^{\infty}\omega (2^{-k}) \int_{(\mathcal {Q}_{k})^{m}}   \dfrac{1}{|2^{k}\sqrt{n} Q |^{m-\alpha/n}} \prod_{j=1}^{m}| f_{j}^{\infty}(y_{j})|  d\vec{y}      \\
&\le   C \sum_{k=1}^{\infty}\omega (2^{-k})  |2^{k+3}\sqrt{n}Q|^{\alpha/n}  \prod_{j=1}^{m}  \dfrac{1}{|2^{k+3}\sqrt{n}Q|}\int_{2^{k+3}\sqrt{n}Q}| f_{j}(y_{j})|  dy_{j}   \\
 &\le C |\omega|_{\dini(1)}   \mathcal{M}_{\alpha}(\vec{f})(x).
\end{align*}

 Now, for $(\rho_{1},\dots,\rho_{m})\in \rho$, let us consider the terms  \eqref{inequ:dini-frac-max}  such that at least one
 $\rho_{j}=0$ and  one  $\rho_{i}=\infty$. Without loss of generality, we assume that  $ \rho_{1}=\cdots=\rho_{\ell}=0$ and $\rho_{\ell+1}=\cdots=\rho_{m}=\infty$ with $1\le \ell<m$. Thus
\begin{align} \label{inequ:dini-frac-max-1}
&  \Big|T_{\alpha}(f_{1}^{\rho_{1}},\dots, f_{m}^{\rho_{m}})(z)   -     T_{\alpha}(f_{1}^{\rho_{1}},\dots, f_{m}^{\rho_{m}})(x)   \Big|  \le  \int_{(\mathbb{R}^{n})^{m}}   |K_{\alpha}(z,\vec{y})-K_{\alpha}(x,\vec{y})| \prod_{j=1}^{m}| f_{j}^{\rho_{j} }(y_{j})|  d\vec{y}    \notag \\
&\le   \int_{( Q^{*})^{\ell}}\prod_{j=1}^{\ell}| f_{j}^{0}(y_{j})| \int_{(\mathbb{R}^{n}\setminus  Q^{*})^{m-\ell}}   |K_{\alpha}(z,\vec{y})-K_{\alpha}(x,\vec{y})| \prod_{j=\ell+1}^{m}| f_{j}^{\infty}(y_{j})|  d\vec{y}    \\
&\le   \int_{(Q^{*})^{\ell}}\prod_{j=1}^{\ell}| f_{j}^{0}(y_{j})| \sum_{k=1}^{\infty}\int_{(\mathcal {Q}_{k})^{m-\ell}}   |K_{\alpha}(z,\vec{y})-K_{\alpha}(x,\vec{y})| \prod_{j=\ell+1}^{m}| f_{j}^{\infty}(y_{j})|  d\vec{y}.         \notag
\end{align}
Since for  $x, z\in Q$, and any $ y_{j}  \in  \mathcal {Q}_{k} $ with $\ell+1\le j \le m$,   there has $2^{k}\sqrt{n} l(Q)\le |z-y_{j}|$,
then, similar to  \eqref{equ:w-CZK-frac-regularity-1-1}, we obtain that
\begin{align*}
|K_{\alpha}(z,\vec{y})-K_{\alpha}(x,\vec{y})| &\le \dfrac{A }{\Big (\sum\limits_{j=1}^{m}|z-y_{j}| \Big)^{mn-\alpha}} \omega \bigg( \frac{|z-x|}{ \sum\limits_{j=1}^{m}|z-y_{j}|} \bigg)            \\
&\le \dfrac{A }{\Big (\sum\limits_{j=\ell+1}^{m}|z-y_{j}| \Big)^{mn-\alpha}} \omega \bigg( \frac{|z-x|}{ \sum\limits_{j=\ell+1}^{m}|z-y_{j}|} \bigg)            \\
&\le \dfrac{C \omega (2^{-k})}{|2^{k}\sqrt{n}  Q |^{m-\alpha/n}}.
\end{align*}
This together with \eqref{inequ:dini-frac-max-1} gives
\begin{align*}
&  \Big|T_{\alpha}(f_{1}^{\rho_{1}},\dots, f_{m}^{\rho_{m}})(z)   -     T_{\alpha}(f_{1}^{\rho_{1}},\dots, f_{m}^{\rho_{m}})(x)   \Big|  \\
& \le C \int_{( Q^{*})^{\ell}}\prod_{j=1}^{\ell}| f_{j}^{0}(y_{j})| \sum_{k=1}^{\infty}\int_{(\mathcal {Q}_{k})^{m-\ell}}   |K_{\alpha}(z,\vec{y})-K_{\alpha}(x,\vec{y})| \prod_{j=\ell+1}^{m}| f_{j}^{\infty}(y_{j})|  \mathrm{d}\vec{y}  \\
& \le C \int_{( Q^{*})^{\ell}}\prod_{j=1}^{\ell}| f_{j}^{0}(y_{j})| \sum_{k=1}^{\infty} \omega (2^{-k})\int_{(\mathcal {Q}_{k})^{m-\ell}}   \dfrac{1}{|2^{k}\sqrt{n} Q |^{m-\alpha/n}} \prod_{j=\ell+1}^{m}| f_{j}^{\infty}(y_{j})|  \mathrm{d}\vec{y}  \\
& \le C \sum_{k=1}^{\infty} \omega (2^{-k})\dfrac{1}{|2^{k}\sqrt{n} Q |^{m-\alpha/n}} \Big(\prod_{j=1}^{\ell}\int_{ Q^{*}}| f_{j}^{0}(y_{j})| \mathrm{d}y_{j} \Big) \Big( \prod_{j=\ell+1}^{m}\int_{2^{k+3}\sqrt{n} Q}| f_{j}^{\infty}(y_{j})|  \mathrm{d}y_{j} \Big)   \\
&\le  C \sum_{k=1}^{\infty} \omega (2^{-k}) |2^{k+3}\sqrt{n} Q |^{\alpha/n} \Big(\prod_{j=1}^{m} \dfrac{1}{|2^{k+3}\sqrt{n} Q |}  \int_{2^{k+3}\sqrt{n} Q}| f_{j}(y_{j})|  \mathrm{d}y_{j} \Big)   \\
&\le   C |\omega|_{\dini(1)}   \mathcal{M}_{\alpha}(\vec{f})(x).
\end{align*}
So, \eqref{inequ:dini-frac-max} is proven.
Therefore, we have
\begin{align*}
 I_{2}  &=\left(\dfrac{1}{|Q|}\dint_{Q} \Big| \sum_{(\rho_{1},\dots,\rho_{m})\in \rho} T_{\alpha}(f_{1}^{\rho_{1}},\dots,  f_{m}^{\rho_{m}})(z) -c \Big|^{\delta} dz \right)^{1/\delta}    \\
&\le C  \mathcal{M}_{\alpha}(\vec{f})(x).
\end{align*}

Combining the above estimates we get the desired result. The proof is completed.
\end{proof}

In 2014, Grafakos et al. \cite{grafakos2014multilineara}  proved the following result in the context of RD-spaces, which serves as an analog of the classical Fefferman-Stein inequalities (see Lemma 4.11 in \cite{grafakos2014multilineara}). Here, we rewrite their result as follows.
\begin{lemma} \label{lem:one-weight-sharp-max}
Let $0< q_{0}\le q <\infty$ and $w\in A_{\infty}$. Then  there exists a  positive constant $C$  (depending on $n$, $q$ and the $A_{\infty}$ constant of $w$), such that for all functions $f\in L_{\loc}^{1}(\mathbb{R}^{n})$ with $M(f) \in L^{q_{0},\infty}(w)$,
\begin{enumerate}[leftmargin=2em,label=(\roman*),itemindent=1.5em]  
\item when $q_{0}< q$,  we have
\begin{align*}
  \|M (f)\|_{L^{q}(w)} &\le  C \| M^{\sharp}(f)\|_{L^{q}(w)}
\end{align*}
\item   when  $q_{0}\le q$,    we have
\begin{align*}
  \|M (f)\|_{L^{q,\infty}(w)} &\le  C \| M^{\sharp}(f)\|_{L^{q,\infty}(w)}.
\end{align*}
\end{enumerate}
\end{lemma}

Now, by \cref{thm.dini-CZ-sharp-max-fract,lem:one-weight-sharp-max,thm:dini-multi-fract-endpoint}, we can get the following result. Since the argument is almost the same as the proof of Proposition 4.13 in \cite{grafakos2014multilineara} (see also Theorem 6.2 in \cite{lu2014multilinear}), we omit the proof.
\begin{theorem} \label{thm:weight-multi-czo-max-fract}
Let $m\ge 2$, $0<\alpha<mn$, $T_{\alpha}$ be an $m$-linear  $\omega_{\alpha}$-CZO with   $\omega \in \dini(1)$,  $n/(mn-\alpha)\le q <\infty$ and $w\in A_{\infty}$. Then  for all bounded functions $\vec{f}$ with compact support, there exists a constant $C>0 $  (depending on $n$, $q$ and the $A_{\infty}$ constant of $w$),
\begin{enumerate}[leftmargin=2em,label=(\roman*),itemindent=1.5em]  
\item when  $n/(mn-\alpha)< q$,   we have
\begin{align*}
  \|T_{\alpha}(\vec{f})\|_{L^{q}(w)} &\le  C \| \mathcal{M}_{\alpha}(\vec{f})\|_{L^{q}(w)}
\end{align*}
\vspace{-2em}
\label{enumerate:weight-multi-czo-max-fract-1}
\item   when  $n/(mn-\alpha)\le q$,   we have
\begin{align*}
  \|T_{\alpha}(\vec{f})\|_{L^{q,\infty}(w)} &\le  C \|\mathcal{M}_{\alpha}(\vec{f})\|_{L^{q,\infty}(w)}.
\end{align*}
\vspace{-2em}
\label{enumerate:weight-multi-czo-max-fract-2}
\end{enumerate}
\end{theorem}
The weighted norm inequalities for $\mathcal{M}_{\alpha}$  were   established in \cite{moen2009weighted} ( see also  \cite{chen2010weighted} or \cite{pradolini2010weighted}).
\begin{lemma}[Weighted estimates for  $\mathcal{M}_{\alpha}$ ] \label{lem.weight-multi-frac-max-I}
Suppose that $0<\alpha<mn $ and $1\le p_{_{1}},p_{_{2}},\dots,p_{_{m}} < \infty$ with  $1/p =  1/p_{_{1}}  + 1/p_{_{2}} +\cdots+1/p_{_{m}}$, such that
$1/q =  1/p  - \alpha/n >0$ and $1/m <p\le q<\infty$.
\begin{enumerate}[leftmargin=2em,label=(\arabic*),itemindent=1.5em]  
\item  If   $1< p_{_{j}} < \infty$ for all $j=1,\dots,m$. Then  $\vec{w}  \in A_{\vec{P},q}$ if and only if 
    $\mathcal{M}_{\alpha}$ can be extended to a bounded operator
    \begin{align*}
  \|\mathcal{M}_{\alpha}(\vec{f})\|_{L^{q}(v_{\vec{w}}^{q})}  &\le C  \prod_{j=1}^{m}\|f_{j}\|_{L^{p_{j}}(w_{j}^{p_{j}})}.
   \end{align*}
\item  If  $1 \le p_{_{j}} < \infty$ for all $j=1,\dots,m$, and at least one $p_{_{j}} =1$ for some $j=1,\dots,m$. Then, for  $\vec{w}  \in A_{\vec{P},q}$,   there is a constant $C > 0$ independent of $\vec{f}$ such that
\begin{align*}
 \|\mathcal{M}_{\alpha}(\vec{f})\|_{L^{q,\infty}(v_{\vec{w}}^{q})}  &\le  C \prod_{j=1}^{m}\|f_{j}\|_{L^{p_{j}}(w_{j}^{p_{j}})}.
\end{align*}
\end{enumerate}
\end{lemma}

\begin{proof} {\bf of \cref{thm:dini-multi-fract-czo-weight}}
Similar to the  reason as in the proof of Theorem 1.2 in \cite{lu2014multilinear},
it is enough to prove \cref{thm:dini-multi-fract-czo-weight} is valid for
$\vec{f}$ being bounded functions with compact supports. By \labelcref{enumerate:multiple-weights-fract} in \cref{lem:multiple-weights}, for   $\vec{w}  \in A_{\vec{P},q}$ there has $ v_{\vec{w}}^{q}  \in A_{\infty} $ . Then \cref{thm:dini-multi-fract-czo-weight} follows from \cref{thm:weight-multi-czo-max-fract,lem.weight-multi-frac-max-I} (the weighted boundedness  of $\mathcal{M}_{\alpha}$ with multiple-weights).

\end{proof}

\section{Proof of \cref{thm:dini-multi-fract-czo-vari}}

The following result, an extrapolation theorem originally due to Cruz-Uribe et al. \cite{cruz2006theboundedness}, are also necessary. Here, we use the following form, see Theorem 7.2.1 in \cite{diening2011lebesgue}.
\begin{lemma}[Extrapolation theorem]\label{lem.extrapolation-thm}
Let $\mathcal{F}$ denote a family  of ordered pairs of measurable functions $(f,g)$. Suppose that for some fixed $q_{0}$ with $0< q_{0}< \infty$  and every weight $w \in A_{1}$ such that
\begin{align*}
\dint_{\mathbb{R}^{n}} |f(x)|^{q_{0}} w(x)\mathrm{d}x     &\le C_{0} \dint_{\mathbb{R}^{n}} |g(x)|^{q_{0}} w(x)\mathrm{d}x .
\end{align*}
Let $q(\cdot)\in \mathscr{P} (\mathbb{R}^{n})$ with $q_{0} \le q_{-}$. If $ (q(\cdot)/q_{0})'\in \mathscr{B}(\mathbb{R}^{n})$, then there exists a positive constant $C$, such that for all $(f,g) \in \mathcal{F}$,
\begin{align*}
\|f \|_{q(\cdot)}    &\le C \|g  \|_{q(\cdot)}.
\end{align*}
\end{lemma}

The following result   was given in  \cite{capone2007fractional} (see Theorem 1.3  in \cite{capone2007fractional} for details).
\begin{lemma}  \label{lem.frac-max-variable}
Let $0<\alpha<n $ and $ p(\cdot), q(\cdot) \in  \mathscr{P}^{\log}(\mathbb{R}^{n})$ with $p_{+}< n/\alpha$ and  $\frac{1}{q(\cdot)}= \frac{1}{p(\cdot)} - \frac{\alpha}{n}$. Then
    \begin{align*}
  \| M_{\alpha}(f)\|_{ q(\cdot)}  &\le C   \|f\|_{ p(\cdot)}.
   \end{align*}
\end{lemma}

In the setting of classical Lebesgue spaces, \cref{lem.frac-max-variable} follows immediately from the boundedness of the Hardy-Littlewood maximal operator. In fact, using H\"{o}lder's inequality it is straightforward to show that
\begin{align*}
M_{\alpha}(f)(x)       &\le \| f\|_{L^{p}(\mathbb{R}^{n})}^{1-p/q} M(f)(x)^{p/q} , \ \ x\in  \Omega.
\end{align*}

We also need the following density property ( see Theorem 3.4.12 in \cite{diening2011lebesgue}).
\begin{lemma}  \label{lem.density-property}
If $  q(\cdot) \in  \mathscr{P} (\mathbb{R}^{n})$, then   $ C_{c}^{\infty}(\mathbb{R}^{n})$  is dense in $L^{q(\cdot)}(\mathbb{R}^{n})$.
\end{lemma}

Now, we have all the ingredients to prove \cref{thm:dini-multi-fract-czo-vari}.
\begin{proof} {\bf of \cref{thm:dini-multi-fract-czo-vari}}
 Since $ q(\cdot) \in \mathscr{B}(\mathbb{R}^{n})$ then, by \cref{lem.variable-property-B}, there exists a $q_{0}$ with $1 < q_{0} <q_{-}$ such that $ (q(\cdot)/q_{0})'\in \mathscr{B}(\mathbb{R}^{n})$. On the other hand, by \cref{thm:weight-multi-czo-max-fract} we see that, for this $q_{0}$ and any $w \in A_{1}$,
\begin{align*}
\dint_{\mathbb{R}^{n}} |T_{\alpha}(\vec{f})(x)|^{q_{0}} w(x)\mathrm{d}x     &\le C  \dint_{\mathbb{R}^{n}} |\mathcal{M}_{\alpha}(\vec{f})(x)|^{q_{0}} w(x)\mathrm{d}x
\end{align*}
holds for   all $m$-tuples $\vec{f} = ( f_{_{1}},f_{_{2}},\dots,f_{_{m}})$  of bounded functions with compact support.

Apply \cref{lem.extrapolation-thm}  to the pair  $(T_{\alpha}(\vec{f}),\mathcal{M}_{\alpha}(\vec{f}))$ and obtain
\begin{align}\label{inequ:multi-maximal-czo-fract}
 \|T_{\alpha}(\vec{f})\|_{q(\cdot)}    &\le C \|\mathcal{M}_{\alpha}(\vec{f})  \|_{q(\cdot)}.
\end{align}

Let $0< \alpha_{1},\dots, \alpha_{m}<n$ with $\alpha = \alpha_{1}+\cdots+ \alpha_{m}$ and $p_{i}(\cdot)<n/\alpha_{i}$ such that $\frac{1}{q_{i}(\cdot)}= \frac{1}{p_{i}(\cdot)} - \frac{\alpha_{i}}{n}~(i=1,\dots,m)$ and $\frac{1}{q(\cdot)} = \frac{1}{q_{1}(\cdot)}  +\cdots+ \frac{1}{q_{m}(\cdot)}$.
By \cref{def.mul-max-frac},  it is easy to see that (see P.89 in \cite{moen2009linear})
\begin{align*}
 \mathcal{M}_{\alpha}(\vec{f})(x) &\le \prod_{i=1}^{m} M_{\alpha_{i}} f_{i}(x)   \ \ \text{for} \ x\in \mathbb{R}^{n}.
\end{align*}
This, together with \eqref{inequ:multi-maximal-czo-fract}, \cref{lem.frac-max-variable} and the generalized H\"{o}lder's inequality (see \cref{lem.holder-inequality}), yields
\begin{align*}
 \|T_{\alpha}(\vec{f})\|_{q(\cdot)}    &\le C \prod_{i=1}^{m} \|M_{\alpha_{i}} f_{i}  \|_{q_{i}(\cdot)}\le C \prod_{i=1}^{m} \| f_{i}  \|_{p_{i}(\cdot)}.
\end{align*}

Now, we have showed that \cref{thm:dini-multi-fract-czo-vari} is valid for all bounded functions $  f_{_{1}},f_{_{2}},\dots,f_{_{m}} $ with compact support. \cref{lem.density-property} concludes the proof of \cref{thm:dini-multi-fract-czo-vari}.

\end{proof}

 \subsubsection*{Funding information:} \noindent
 This work is  financially supported by the Science and Technology Fund of Heilongjiang  (No.2019-KYYWF-0909),  the NNSF of China (No.11571160), the Reform and Development Foundation for Local Colleges and Universities of the Central Government (No.2020YQ07) and the Scientific Research Fund of MNU (No.D211220637). 

 \subsubsection*{Conflict of interest: } \noindent
The authors state that there is no conflict of interest.   

 \subsubsection*{Data availability statement:} \noindent
  All data generated or analysed during this study are included in this manuscript. 
 \subsubsection*{Author contributions:}\noindent
 All authors contributed equally to   this work. 
 All authors read the final manuscript and approved its submission.


\phantomsection
\addcontentsline{toc}{section}{References}
\bibliographystyle{tugboat}          
\addtolength{\itemsep}{-0.5 em} 
\bibliography{wu-reference} 

\end{document}